%% file: commensurator-arXiv3.tex
\documentclass[12pt,a4paper]{amsart}
\makeatletter
\renewcommand\normalsize{%
	\@setfontsize\normalsize{11.7}{14pt plus .3pt minus .3pt}%
	\abovedisplayskip 10\p@ \@plus4\p@ \@minus4\p@
	\abovedisplayshortskip 6\p@ \@plus2\p@
	\belowdisplayshortskip 6\p@ \@plus2\p@
	\belowdisplayskip \abovedisplayskip}
\renewcommand\small{%
	\@setfontsize\small{9.5}{12\p@ plus .2\p@ minus .2\p@}%
	\abovedisplayskip 8.5\p@ \@plus4\p@ \@minus1\p@
	\belowdisplayskip \abovedisplayskip
	\abovedisplayshortskip \abovedisplayskip
	\belowdisplayshortskip \abovedisplayskip}
\renewcommand\footnotesize{%
	\@setfontsize\footnotesize{8.5}{9.25\p@ plus .1pt minus .1pt}
	\abovedisplayskip 6\p@ \@plus4\p@ \@minus1\p@
	\belowdisplayskip \abovedisplayskip
	\abovedisplayshortskip \abovedisplayskip
	\belowdisplayshortskip \abovedisplayskip}
\ifdefined\pdfpagewidth
\else
\fi
\calclayout
\makeatother
\usepackage{import}
\usepackage{graphicx}
\usepackage{amsmath,amsfonts,amssymb,amsthm}
\usepackage{pstricks,calc} 
\usepackage{srcltx}
\usepackage{comment}
\usepackage{hyperref}
\usepackage[all]{xy}
\usepackage[utf8]{inputenc} 
\usepackage[T1]{fontenc} 
\usepackage{enumerate}

\newtheorem{thmA}{Theorem}

\newtheorem{definition}{Definition} [section]

\newtheorem{theorem}[definition]{Theorem} 

\newtheorem{proposition}[definition]{Proposition}

\newtheorem{lemma}[definition]{Lemma}
\newtheorem{corollary}[definition]{Corollary}

\newtheorem{remark}[definition]{Remark}

\def\F{{\mathcal{F}}}
\def\FA{{\mathcal{FA}_N}}

\def\co{\colon}

\def\aut{{\rm{Aut}}}
\def\IA{{\rm{IA}}}
\def\saut{{\rm{SAut}}}
\def\autN{{\rm{Aut}}(F_N)}
\def\out{{\rm{Out}}}
\def\G{\Gamma}

\def\<{\langle}
\def\>{\rangle}

\def\2rose{{ {$(N;2)$--rose}}}

\newcommand{\inn}{{\rm{Inn}}}
\newcommand{\innn}{{\rm{Inn}}(F_N)}
\newcommand{\autn}{{\rm{Aut}}(F_N)}
\newcommand{\Out}{{\rm{Out}}}
\newcommand{\Comm}{{\rm{Comm}}}

\newcommand{\Aut}{{\rm{Aut}}}
\def\Saut{{\rm{SAut}}}

\newcommand{\ad}{{\rm{ad}}}

\newcommand{\calx}{\mathcal{X}}

\newcommand{\lk}{{\rm{lk}}}

\newcommand{\ian}{{\rm{IA}}_N}

\newcommand{\E}{\mathbb{E}} 
\def\Z{\mathbb{Z}}

\def\R{\mathbb{R}}
\def\FA{{\mathcal{AF}_N}}

\def\Lk{{\rm{lk}}}

\def\ssm{\smallsetminus}

\newcommand{\psmod}{{\rm{PMod}} }
\newcommand{\smod}{{\rm{Mod}}}
\newcommand{\pmodd}{{\rm{PMod}^\pm}}
\newcommand{\modd}{{\rm{Mod}^\pm}}

\title{Commensurations of $\Aut(F_N)$ and its Torelli subgroup}
\author{Martin R. Bridson and Richard D. Wade}
				
\usepackage{hyperref}

\begin{document}

\maketitle

\begin{abstract} 
For $N \geq 3$, the abstract commensurators of both $\aut(F_N)$ and its Torelli subgroup $\ian$ are isomorphic to $\aut(F_N)$ itself. 
\end{abstract}

\section{Introduction}

{\let\thefootnote\relax\footnotetext{{\em Subject} [2020 MSC]. Primary 20F65, 20E36; Secondary 20E05, 57M07. }}
{\let\thefootnote\relax\footnotetext{{\em Keywords}.  {Automorphism groups of free groups, abstract commensurator, free factor complex}}}

A \emph{commensuration} of a group $G$ is an isomorphism  
between two finite-index subgroups of $G$. We consider two commensurations to be equivalent if they agree on a common finite-index subgroup of their domains. The set of all commensurations with this equivalence relation forms a group $\Comm(G)$ called the \emph{abstract commensurator} of $G$. There is a natural homomorphism $\ad \co G \to \Comm(G)$ induced by conjugation. If this map is surjective, then $G$ is said to be
{\em commensurator-rigid}. 

Arithmetic lattices in semisimple Lie groups are never commensurator-rigid, 
but deep results of Mostow \cite{mostow} and Margulis \cite{margulis} tell us that 
non-arithmetic lattices in groups other than ${\rm{SL}}(2,\R)$
will have finite index in their abstract commensurators, and will be commensurator-rigid if they are maximal. The arithmetic case provides a point of contrast in the
rich 3-way analogy between automorphism groups of free groups, lattices such as ${\rm{SL}}(n,\Z)$, and mapping
class groups of surfaces of finite type \cite{Best-icm, BV06}: Ivanov \cite{Iva} proved that (with some exceptions for
small surfaces) the extended mapping class group  
of a surface of finite type is commensurator-rigid, while Farb and Handel \cite{FH} (for $N\ge 4$)
and Horbez and Wade \cite{HW2} (for $N\ge 3$) proved that the same is true for $\out(F_N)$,
the outer automorphism group of a free group (see \cite{Stu} for a survey of the area).  

Our main purpose in this article is to prove that $\aut(F_N)$ is also commensurator rigid for $N\ge 3$. (There is no obvious formalism by which commensurator rigidity for $\aut(F_N)$ can be deduced from commensurator rigidity for $\out(F_N)$, or \emph{vice versa}.)
In the case of mapping class groups and $\out(F_N)$, non-trivial variations on the proof of commensurator rigidity allow one to prove that 
the ambient group  
is also the abstract commensurator of subgroups 
that are big in a suitable sense, e.g. \cite{HW2, BM2, BM, BPS}. 
Our second purpose in this article is to explore the extent to which this is also true in $\aut(F_N)$.
Our previous work \cite{BW} on embeddings of direct products of free groups into $\aut(F_N)$ will play a crucial role in 
both halves of the paper. The idea of using such embeddings to
explore commensurators originates in \cite{BPS} and was used in \cite{HW2}.

\begin{thmA}\label{t:n}
For all $N\ge 3$, 
the natural map $\ad \co \aut(F_N)\to {\rm{Comm}}(\aut(F_N))$ is an isomorphism. More precisely, any isomorphism $f\co \G_1 \to \G_2$ between finite-index subgroups
of $\aut(F_N)$ is the restriction of conjugation by an element of $\aut(F_N)$.
\end{thmA}

Our proof of this theorem will also allow us to identify the abstract commensurator of the subgroups of $\aut(F_N)$ that are {\em ample} in the following sense. 
Recall that a \emph{Nielsen automorphism} is one that takes $x_1 \mapsto x_1x_2$ and fixes all other basis elements for some basis $x_1, \ldots, x_N$ of $F_N$.
Recall too that an element is \emph{primitive} if it belongs to some basis of $F_N$.

\begin{definition}[Ample subgroups]
A subgroup $\G < \aut(F_N)$ is \emph{ample} if $\G$ contains a nontrivial power of every Nielsen automorphism (hence a nontrivial power of the inner automorphism determined by each primitive element).
\end{definition}

All finite-index subgroups of $\Aut(F_N)$ are ample. Examples of ample, infinite-index subgroups include the kernels of the natural maps from $\aut(F_N)$ to 
the automorphism groups of free Burnside groups \cite{MR3134411}. 
The Torelli subgroup $\IA_N< \Aut(F_N)$ is not ample (any nontrivial power of a Nielsen automorphism acts nontrivially on $H_1(F_N)$);
we will elucidate its commensurator using different arguments later in the paper. 

To describe the full theorem in the ample case precisely, recall that the \emph{relative commensurator} of a subgroup $\G < G$ is the group \[ \Comm_G(\G) =\{ g \in G : \text{$\G \cap g\G g^{-1}$ and $\G \cap g^{-1}\G g$ have finite-index in $\G$}\}. \] If $\G$ is normal in $G$, then its relative commensurator is $G$ itself. In general, the relative commensurator of $\G<G$ lies between its normalizer $N_G(\G)$ and $G$. The action of $G$ on $\G$ by conjugation induces a homomorphism from 
$\Comm_G(\G)$ to the abstract commensurator of $\G$. For ample subgroups of $\aut(F_N)$, we show that this map is an isomorphism.

\begin{thmA}\label{t:n2}
If $\G$ is an ample subgroup of $\aut(F_N)$ then the natural map \[\Comm_{\aut(F_N)}(\G)\to \Comm(\G)\] is an isomorphism. Furthermore, any isomorphism $f\co \G_1 \to \G_2$ between finite-index subgroups of $\G$ is the restriction of conjugation by an element of $\aut(F_N)$.
\end{thmA}

For a subgroup $\G$ of a group $G$, Ivanov \cite{Iva} (building on ideas of Tits \cite{tits} and Margulis \cite{margulis})
pioneered the following approach to proving commensurator rigidity theorems:

\begin{itemize}
\item Find a complex $X$ which is $G$-rigid (its isometry group is exactly $G$).
\item Show that the action of $\G < G$ on $X$ extends to an action of $\Comm(\G)$.
\item A standard exercise (see Lemma~\ref{l:Ivanov}) then implies that the abstract commensurator of $\G$ is equal to its relative commensurator in $G$.
\end{itemize}

The first candidate complex that comes to mind is the \emph{free factor graph} $\mathcal{AF}_N$, which was shown to be rigid by Bestvina and Bridson \cite{BB}. 
In Section \ref{s:fat_action} we explain how our previous work on direct products in $\aut(F_N)$ allows one to prove that, so long as $\G < \aut(F_N)$ is ample, the action of $\G$ on the free factors of rank one and two extends to an action of $\Comm(\G)$. 
The key result we use here is that powers of Nielsen automorphisms are the \emph{only} automorphisms that centralize subgroups of the form $F_2^{2N-3}$ in $\aut(F_N)$, and so such automorphisms are preserved under commensurations of ample subgroups. A `typical' copy of $F_2^{2N-3} < \aut(F_N)$ has one direct factor consisting of inner automorphisms from a rank two free factor $A<F_N$, while all the other factors contain many Nielsen automorphisms. As powers of Nielsen automorphisms are preserved by $\Comm(\G)$, we can use this observation to induce an action of $\Comm(\G)$ on the set of rank two free factors. Looking at intersections of these subgroups allows us to also keep track of the effect of commensurations on rank one free factors, and thus we obtain an action of $\Comm(\G)$ on the subgraph $\mathcal{F}_{(2)}$ of the free factor graph spanned by rank one and two free factors. However, the action of $\Comm(\G)$ on free factors of rank greater than two is harder to obtain directly. We instead prove the following:

\begin{thmA} \label{t:f2_is_rigid}
Let $\mathcal{F}_{(2)}$ be the subgraph of the free factor graph $\mathcal{AF}_N$ spanned by free factors of rank one and two. Then, every  
automorphism of $\mathcal{F}_{(2)}$ extends to a unique automorphism of $\mathcal{AF}_N$, hence \[ \aut(\mathcal{F}_{(2)}) \cong \aut(\mathcal{AF}_N). \]  
 \end{thmA}

The proof, which draws on ideas from \cite{BB}, proceeds one layer at a time, extending automorphisms of $\mathcal{F}_{(2)}$ to automorphisms of $\mathcal{F}_{(3)}$ (the subgraph that also includes rank 3 free factors) and continuing upwards. A key part of the proof of the rigidity of $\mathcal{AF}_N$ in \cite{BB} is the \emph{Antipode Lemma}, which metrically distinguishes \emph{antipodes}  (pairs of vertices that together generate higher rank free factors) in the free factor graph. We have less information in our setting as we are working in a subgraph of $\mathcal{AF}_N$, so the first step is to prove a version of the Antipode Lemma that can be used in the subgraphs $\mathcal{F}_{(k)}$ spanned by factors of rank at most $k$. This is called the \emph{Graded Antipode Lemma}, and is the focus of Section~\ref{s:aal}. The method of proof is a combinatorial argument using Whitehead graphs (whereas the proof of the original Antipode Lemma  \cite{BB} relied on dynamical properties of fully irreducible automorphisms).

The Graded Antipode Lemma allows us to distinguish, in an $\aut(\mathcal{F}_{(2)})$-invariant way, pairs of vertices that generate a rank 3 factor in $\mathcal{F}_{(2)}$. But in order to extend an automorphism of $\mathcal{F}_{(2)}$ to an automorphism of $\mathcal{F}_{(3)}$, we also need to tell when two pairs of antipodes generate the same rank 3 factor using only the metric properties of the graph.  This is the content of Section~\ref{s:extending}, and it is achieved by first showing that $\aut(\mathcal{F}_{(2)})$ preserves the set of \emph{standard 3-apartments}, that is subgraphs spanned by factors corresponding to the subsets of a fixed basis for a rank 3 factor. We then argue that two such apartments $\Delta$ and $\Lambda$ can be connected by a chain of apartments $\Delta=\Delta_0,\Delta_1, \ldots, \Delta_n=\Lambda$ 
in which neighbours share an antipodal pair if only if the vertices of $\Delta$ and $\Lambda$ generate the same rank 3 factor. With this metric characterisation in hand,  we have a canonical way of extending automorphisms of $\mathcal{F}_{(2)}$ to automorphisms of $\mathcal{F}_{(3)}$. A similar argument lets us extend automorphisms of $\mathcal{F}_{(k)}$ to automorphisms of $\mathcal{F}_{(k+1)}$, enabling us to prove Theorem~\ref{t:f2_is_rigid} by induction.

\subsection{The Torelli subgroup and failure of the approach for lower terms of the Andreadakis--Johnson filtration.}

Recall that the \emph{Torelli subgroup} of $\ian$ of $\aut(F_N)$ is the subgroup acting trivially on $H_1(F_N)=F_N/[F_N,F_N]$. As mentioned above, the Torelli subgroup is not ample, however it is still commensurator-rigid:

\begin{thmA} \label{t:torelli}
For all $N\ge 3$, 
the natural map $\aut(F_N)\to {\rm{Comm}}(\ian)$ is an isomorphism. Furthermore, any isomorphism 
 between finite-index subgroups of $\ian$ is the restriction of conjugation by an element of $\aut(F_N)$.
\end{thmA}

We again deduce rigidity by constructing an action of $\Comm(\ian)$ on $\mathcal{F}_{(2)}$, however we need to use a different method to the ample case. Rather than taking the most general approach possible, we use a property that is specific to subgroups $\G$ commensurable with $\ian$: in this case, for a maximal direct product of free groups $D< \Gamma$ (where maximal means with respect to both the number of factors and containment), our classification theorem in \cite{BW} implies that the factor of $D$ containing   inner automorphisms   is the unique direct factor which is finitely generated. This property allows us to obtain an action of $\Comm(\ian)$ on the rank two factors and, again, by using intersections, we can  extend this action to the set of rank one factors. Thus we obtain an action on $\mathcal{F}_{(2)}$
and can appeal to Theorem \ref{t:f2_is_rigid}.

The Torelli subgroup is the first in a descending chain of subgroups of $\aut(F_N)$ called the \emph{Andreadakis--Johnson filtration} (introduced in \cite{MR188307}), defined as follows: the $k$th subgroup in the series is the kernel of the map \[ \aut(F_N) \to \aut(F_N/\gamma_{k+1}(F_N)), \] where $\gamma_{k+1}(F_N)$ is the $k+1$ term of the lower central series of $F_N$. While it seems likely that all of these groups are commensurator rigid (indeed this is known to be the case for $\out(F_N)$ \cite{HW2}), the methods used in this paper are not immediately applicable to these subgroups. In more detail, if $\G$ has finite index in the Torelli subgroup or else $\G$ is  ample, then $\G$ contains a nontrivial power of every inner automorphism induced by a primitive element; it is these inner automorphisms that we use to control the action of $\Comm(\G)$ on the rank one free factors. Such automorphisms are not contained in higher terms of the Andreadakis--Johnson filtration and it is not clear to us how to regain control over the rank one free factors when
working with these subgroups. It is plausible that one needs to use a different rigid $\aut(F_N)$--graph to prove commensurator rigidity for terms further down the filtration.

\subsection{Low-rank cases}

As $\aut(F_1)\cong \mathbb{Z}/2\mathbb{Z}$, its abstract commensurator is trivial. The abstract commensurator of $\aut(F_2)$ is more interesting:

\begin{thmA}\label{t:comm_f2}
The abstract commensurator of $\aut(F_2)$ is isomorphic to the extended mapping class group of a sphere with five punctures. Furthermore, $\aut(F_2)$ is an index-five subgroup of its abstract commensurator.
\end{thmA}

\begin{proof}
Dyer, Formanek, and Grossman show that the unique index-two subgroup $\saut(F_2)$ of $\aut(F_2)$ (the kernel of the determinant map for the action on homology) is isomorphic to the quotient of the four strand braid group $B_4$ by its centre $Z_4$ \cite[p. 406]{MR666911}. Hence $\Comm(\aut(F_2))\cong \Comm(B_4/Z_4)$. Using work of Korkmaz \cite{MR1696431}, Charney and Crisp proved that $\Comm(B_4/Z_4)$ is isomorphic to the extended mapping class group of a five-punctured sphere \cite[Theorem~1] {MR2150887}. They furthermore show that $B_4/Z_4$ has index 10 in its abstract commensurator, and it follows that $\aut(F_2)\hookrightarrow \Comm(\aut(F_2))$ is of index 5,
since any automorphism fixing a finite-index subgroup of $F_2$ is trivial. 
\end{proof}

In an appendix to this paper we give a geometric proof of the result of Dyer, Formanek, and Grossman. We also explain how the non-trivial commensurations of $\aut(F_2)$ arise. Briefly, it is because $\aut(F_2)$ is the extended pure mapping class group of a twice-punctured torus, and the five-punctured sphere appears as the quotient of this torus by its hyperelliptic involution.

\subsection{A brief outline of the paper}

In Section~\ref{s:aal} we prove the Graded Antipode Lemma and in Section~\ref{s:extending} we use it to prove that the group of 
automorphisms of the graph $\mathcal{F}_{(2)}$ is isomorphic to $\aut(F_N)$ (Theorem~\ref{t:f2_is_rigid}). In Section~\ref{s:fat_action}, for any ample subgroup $\G < \aut(F_N)$ we construct a natural action of $\Comm(\G)$ on $\mathcal{F}_{(2)}$, and we use the rigidity of $\mathcal{F}_{(2)}$ to deduce commensurator rigidity of $\G$ (Theorem~\ref{t:n2}). Commensurator rigidity of $\aut(F_N)$ (Theorem~\ref{t:n}) follows directly from Theorem~\ref{t:n2}. In Section~\ref{s:torelli}, we complete the proof of Theorem~\ref{t:torelli} by showing that the abstract commensurator $\Comm(\ian)$ of the Torelli subgroup $\ian$ also admits a natural action on $\mathcal{F}_{(2)}$.

\subsection{Acknowledgements}

We thank Mladen Bestvina and Camille Horbez for many fruitful conversations surrounding the material in this paper.
We thank Matt Clay for asking about the abstract commensurator of $\aut(F_2)$ and for his suggestions concerning its geometry. We also thank the referee for their careful reading and helpful comments. The second author is supported by a University Research Fellowship from the Royal Society.

\section{Whitehead Graphs and the Graded Antipode Lemma} \label{s:aal}

In the first two sections of  the paper, we show that every automorphism of $\F_{(2)}$ extends to an automorphism of the whole free factor graph. Applying the main theorem of \cite{BB}, this shows that $\mathcal{F}_{(2)}$ is rigid. The first step, and the subject of this first section, is a generalization of Bestvina and Bridson's \emph{Antipode Lemma}, which we shall prove using \emph{Whitehead graphs} for elements of $F_N$.

\subsection{Background on Whitehead graphs and automorphisms}

Fix a basis $X=\{x_1, \ldots, x_n\}$ of $F_N$. Let $\mathcal{X}=X \cup X^{-1}\cup \{o\}$ be the set of letters and their inverses in this basis, along with a distinguished letter $o$ that will appear as a basepoint imminently. It is convenient to write $\bar{x}_i:=x_i^{-1}$. Given an element $u \in F_N$, its \emph{Whitehead graph} is the graph $W(u)$ whose vertex set is $\mathcal{X}$, with an edge joining $x$ to $\bar{y}$ for each pair of consecutive letters $xy$ in $u$ and two additional
edges from the basepoint $o$ to $\overline{s}$ and $t$, where $s$ and $t$ are the initial and terminal letters of $u$,
respectively.

\begin{figure}[ht]
\centering
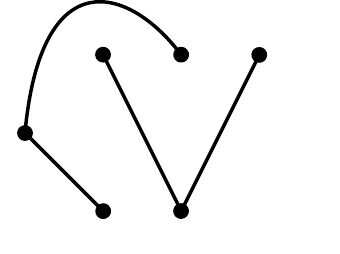
 \caption{The Whitehead graph of $u=x_1x_2x_3x_2$.}\label{f:W1}
\end{figure}

Suppose $C \cup C'$ is a partition of $\calx$ such that $a \in C$ and $\bar{a} \in C'$. If $o \not\in C$ then the associated {\em Whitehead automorphism }$\phi=\phi(C;a)$ fixes $a$ and is defined on the remainder of $X$ by:

\[
\phi(x)=\begin{cases} x & \text{if $x, \overline{x} \in C'$ } \\
xa &\text{if $x \in C$ and $\overline{x} \in C'$} \\
\bar{a}x &\text{if $x \in C'$ and $\bar{x} \in C$} \\
\bar{a}xa &\text{if $x, x' \in C$}
\end{cases}
\]

 If $o \in C$, we alter $\phi$ by an inner automorphism and instead define $\phi$ by:

\[
\phi(x)=\begin{cases} ax\bar{a} & \text{if $x, \overline{x} \in C'$ } \\
ax &\text{if $x \in C$ and $\overline{x} \in C'$} \\
x\bar{a} &\text{if $x \in C'$ and $\bar{x} \in C$} \\
x &\text{if $x, x' \in C$}
\end{cases}
\]

This shift is used by Whitehead \cite{W1} but is not common in the literature, as it is not necessary when dealing with cyclic words (i.e. words up to conjugacy).  Importantly for us, if both $x$ and $\bar{x}$ are on the same side of the partition as $o$, then they are fixed by the associated Whitehead automorphism. We call $a$ the \emph{acting letter} of the automorphism. We say that a connected component of $W(u)$ is \emph{nontrivial} if it is not an isolated vertex. The following lemma is standard.

\begin{lemma}[Basic properties of $W(u)$] \label{l:basic}
For any letter $x$, the valences of $x$ and $\bar{x}$ are equal. If $W(u)$ has two or more nontrivial connected components then each nontrivial component $C$ contains an element $x \in C$ such that $\bar{x} \not \in C$.
\end{lemma}

\begin{definition}
A letter $a \in \mathcal{X}$ is a \emph{reducing letter of $u$} if $a \neq o$ and either
\begin{itemize}
\item $a$ is a cut vertex in its connected component of $W(u)$, else
\item $a$ or $\bar{a}$ appears in $u$ and the connected component of $W(u)$ containing $a$ does not contain $\bar{a}$.
\end{itemize}
\end{definition}

Lemma~\ref{l:basic} tells us that any Whitehead graph with two or more nontrivial connected components contains a reducing letter.

\begin{lemma}[Whitehead \cite{W1}] Let $a \in X \cup X^{-1}$ be a basis element or its inverse. Let $u \in F_N$ and let $W(u)$ be its Whitehead graph.
\begin{itemize}
\item If $C$ is a component of $W(u) \ssm \{a\}$ disjoint from $\bar{a}$ and $\phi=\phi(C\cup\{a\},a)$ then \[ |u| - |\phi(u)| \] is equal to the number of edges between $a$ and elements of $C$. 
\item If $a$ is a reducing letter for $u$ then there exists a Whitehead automorphism $\phi$ with acting letter $a$ such that $|\phi(u)| < |u|$. 
\end{itemize}
\end{lemma}

\begin{proof}[Sketch proof]
This is contained on pages 50-51 of Whitehead \cite{W1}. The first point follows by considering the cancellations introduced by $\phi$ (no new letters are added and we have cancellations for each edge between $a$ and a letter in $C$). For the second point, if $a$ is reducing, then we can find an appropriate component $C$ of $W(u)\ssm \{a\}$ that allows us to apply the second bullet point.
\end{proof}

Whitehead showed that primitive elements that are not part of the fixed  basis $X$ can always be reduced using the process above.

\begin{theorem}[Whitehead's cut-vertex theorem \cite{W1}] \label{t:Whitehead}
If $u \in F_N$ is primitive then either $u \in \mathcal{X}$ or else $u$ contains a reducing letter.
\end{theorem}

For more complicated elements, Whitehead showed that non-minimal elements (in their $\aut(F_N)$-orbit) always have reducing automorphisms \cite{W2}. However, the reduction is not always visible from the connectivity properties of the Whitehead graph (one also has to consider weightings, or the number of edges between points in the graph). Even though the cut-vertex criterion does not give a way of proving minimality of arbitrary elements, it was observed by Stallings (using the same 3--manifold perspective as Whitehead half a century later) that we can at least use these methods to find the free factor support of an element.

\begin{theorem}[Stallings {\cite[Theorem~2.4]{Sta}}] \label{t:Stallings}
If $u \in F_N$ is cyclically reduced and $W(u)$ is connected with no cut vertex, then $u$ is not contained in a proper free factor of $F_N$.
\end{theorem}

\begin{proof}[Sketch proof] As $u$ is cyclically reduced, it is enough to consider the \emph{cyclic Whitehead graph} $W_c(u)$, which is homeomorphic to $W(u)$, but the basepoint is no longer considered to be a vertex. The 3--manifold $M=\#_N S^1 \times S^2$ has fundamental group $F_N$ and we can take a sphere system $S$ dual to the rose of the basis used to define the Whitehead graph. The element $u$ is given by a based loop $l$ in $M$, and we can homotope away segments of the loop $l$ in $M\ssm S$ that have both endpoints  on the same side of one of the spheres (this is equivalent to $l$ tracing out a reduced word). 
After this, the edges in the cyclic Whitehead graph correspond to arcs of $l$ in $M\ssm S$ (as the vertices of the graph correspond to the $2N$ sides of the $N$ spheres). Suppose for a contradiction that $u$ is contained in a proper free factor. Then there is another sphere system $\Sigma$ dual to a different rose (i.e.~basis) such that the reduced form of $u$ does not cross a loop in this rose (or equivalently, a sphere in $\Sigma$). We put $S$ and $\Sigma$ into Hatcher normal form (see \cite{Hat}) and homotope the loop $l$ corresponding to the conjugacy class of $u$ to have minimal intersections with both $S$ and $\Sigma$ (this part is somewhat technical, and the main reason why this proof is a sketch). In this case, there exists $\sigma \in \Sigma$ not crossed by $u$. If $\sigma$ is disjoint from $S$, then the sphere gives a partition of the sides of the spheres in $M\ssm S$, which shows that $W_c(u)$ is disconnected (see Figure~\ref{f:Stallings}). Otherwise, $\sigma \cap S$ is a collection of circles on $\sigma$: pick an innermost such circle bounding a disc $D$ and intersecting a sphere $s \in S$. We claim that the side of $s$ bounding $D$ is a cut vertex for $u$ with respect to the Whitehead graph 
given by $S$. Indeed, the loop $l$ does not cross $D$, so  $D$ partitions the remaining sides of spheres of $S$ into two pieces, which have no connecting edges (although they can be joined by a path through $s$).
\end{proof}

\begin{figure}[h]
\centering
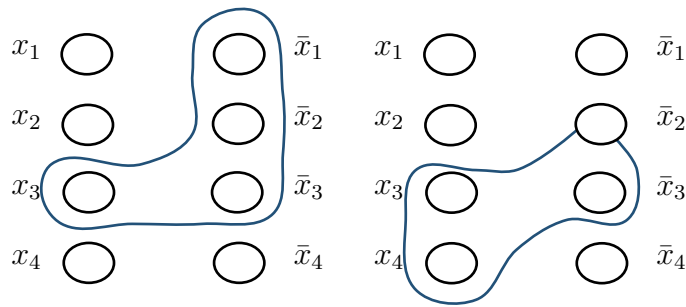
 \caption{A pictorial proof of Stallings' theorem in $F_4$: if $u$ is contained in a proper free factor then 
 the corresponding loop is disjoint from a sphere $\sigma$. If $\sigma$ is disjoint from the spheres $S$ dual to our standard basis then $W(u)$ is disconnected (left), otherwise an innermost circle of intersection on $\sigma$ gives a cut vertex ($\bar{x}_2$ in this example) via its bounding disc $D$ (right). The loops corresponding to $u$ are omitted.}
 \label{f:Stallings}
\end{figure}

\subsection{The Graded Antipode Lemma}

\begin{definition}
Two free factors $A$ and $B$ in $F_N$ are \emph{antipodal} if the subgroup $\langle A,B \rangle < F_N$ they generate is a free factor isomorphic to $A \ast B$. 
\end{definition}

The following is the key technical result for extending automorphisms of $\F_{(2)}$ to automorphisms of $\FA$. 

\begin{theorem}[Graded Antipode Lemma]
Let $A=\langle x_1, \ldots, x_k \rangle$ be the free factor generated by the first $k$ elements of a basis $\{x_1, \ldots, x_N\}$
for $F_N$, with $k \geq 2$. Suppose that $u \in F_N$ is primitive. Then exactly one of the following conditions holds:
\begin{itemize}
\item $u \in A$.
\item There exists an automorphism $\phi$ preserving $A$ such that $\phi(u)=x_{k+1}$.
\item There exists a primitive element $p \in A$ such that $\langle p, u\rangle$ is not contained in a free factor of rank 2.
\end{itemize}
\end{theorem}

\begin{proof}
Suppose $u \not \in A$ and there does not exist an automorphism preserving $A$ such that $\phi(u)=x_{k+1}$. We let $B= \langle x_{k+1}, \ldots, x_N \rangle$ and use $\mathcal{A}$ and $\mathcal{B}$ to denote the set of elements and  inverses for the given bases of $A$ and $B$. We are free to replace $u$ with its image under any automorphism preserving $A$, so we can assume $|u|$ is minimal with respect to such automorphisms. Minimality of $|u|$ in this respect is the key device for getting the arguments in this lemma to work. By a simple induction on rank, we may also assume that $u$ contains every letter of $\{x_{k+1}, \ldots, x_N\}$ or its inverse. Furthermore, after applying a conjugation by an element of $A$, we may assume that the initial letter $s$ of $u$ belongs to $\mathcal{B}$. Let $t$ be the terminal letter of $u$; we allow the possibility that $t \in \mathcal{A}$. 

As $u$ is primitive, the assumption that no automorphism fixing $A$ sends $u$ to $x_{k+1}$ implies that $u$ also contains letters from $\mathcal{A}$. After applying a signed permutation of the basis of $A$, we can furthermore assume:

\begin{itemize}
\item $x_1$ has incident edges in $W(u)$ and is not equal to $\bar{t}$
\item $x_2$ is not equal to $t$.
\end{itemize}

For the first bullet point, we reorder the basis so that $x_1$ or $\bar{x}_1$ is in the support of $u$, and if $x_1 = \bar{t}$ we swap $x_1$ with $\bar{x}_1$. Similarly, we apply the inversion swapping $x_2$ with $\bar{x}_2$ if $x_2=t$. We then let 

\[ p=x_1x_2^2 \cdots x_k^2x_1x_2^2\cdots x_k^2 x_2. \]

The Whitehead graph of $p$ is shown in Figure~\ref{f:W2}.
\begin{figure}[h]
\centering
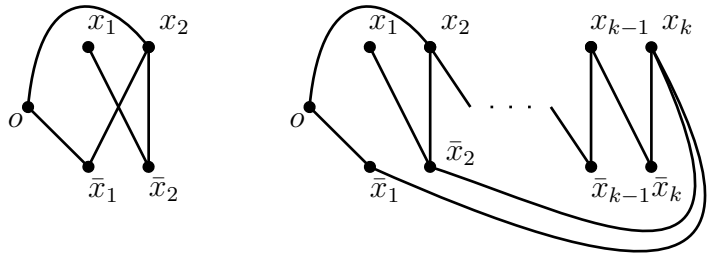
\caption{On the left is the Whitehead graph of $p$ when $k=2$ and $p=x_1x_2^2x_1x_2^3$, and the general case $k \geq 3$ is given on the right. Duplicate edges are omitted. When $k=2$ both $x_2$ and $\bar{x}_2$ are cut vertices, and when $k \geq 3$ only $\bar{x}_2$ is a cut vertex.}\label{f:W2} 
\end{figure}
The element $p$ is the image of $x_2$ under the automorphism \[ \rho_{12}^2 \rho_{13}^2 \cdots \rho_{1k}^2 \lambda_{21}^2, \] where $\lambda_{21}$ is the left Nielsen automorphism   $x_2 \mapsto x_1x_2$ and $\rho_{1j}$ is the right Nielsen automorphism   $x_1 \mapsto x_1x_j$. Hence $p$ is primitive. We let 
\begin{align*} w&=pu\bar{p}up \\ &= x_1 \cdots x_2 s \cdots t \bar{x}_2 \cdots \bar{x}_1 s \cdots t x_1 \cdots x_2.  \end{align*}
Note that $w$ is cyclically reduced in this form, so $W(w)$ is obtained from $W(u)$ and $W(p)$ as follows: take all edges of $W(u)$ and $W(p)$ (with appropriate multiplicity) and throw out the edges from $\bar{s}$ and $t$ to $o$ in $W(u)$, then attach edges from $\bar{s}$ and $t$ to both $\bar{x}_1$ and $x_2$. We will show that $W(w)$ is connected and has no cut vertex, hence by Stallings' theorem, $w$ is not contained in any free factor of $F_N$. To see this, we claim that for any $x \in \mathcal{A} \cup \mathcal{B}$ and any $y \neq x$, there is a path from $y$ to $o$ in $W(w) \ssm x$. 

First suppose that $y \in \mathcal{A}$. As $W(p)$ is a subgraph of $W(u)$, there is a path from $y$ to $o$ in $W(w)\ssm x$ unless $x$ is a cut vertex of $W(p)$, in which case $x=\bar{x}_2$ or $k=2$ and $x=x_2$. In either situation, $W(p)\ssm x$ has two components, one of which contains $o$ and the other contains $x_1$. By the first part of the claim below, there is a path from $x_1$ to either $\bar{x}_1$, $\bar{s}$, or $t$ in $W(u)\ssm\{o,x\}$, and as $\bar{x}_1$ is adjacent to both $\bar{s}$ and $t$ in $W(w)$, the two components of $W(p)\ssm x$ are connected by a path in $W(w)$. 

If $y \in \mathcal{B}$, then $y$ or $\bar{y}$ appears in $u$ and $y$ lies in some nontrivial component of $W(u)$. By Part 4 of the claim below, every nontrivial component of $W(u)\ssm\{o,x\}$ contains either $\bar{s}$, $t$, or a letter $a \in \mathcal{A}$. As $\bar{s}$ and $t$ are adjacent to both $\bar{x}_1$ and $x_2$, we can find a path in $W(w)\ssm x$ to an element of $\mathcal{A}$, and therefore to $o$. Thus the following claim  completes the proof of the lemma. 

\paragraph*{Claim: Key properties of $W(u)$.}

\begin{enumerate}
 \item If $x \in \mathcal{A}$ and $x \neq x_1, \bar{x}_1$, then there exists a path from $x_1$ to either $\bar{x}_1$, $\bar{s}$, or $t$ in $W(u)\ssm\{o,x\}$.
 \item Every nontrivial component of $W(u)$ contains either $o$ or an element of $\mathcal{A}$
 \item If $x \in \mathcal{A} \cup \mathcal{B}$, then every nontrivial component of $W(u)\ssm x$ contains either $o$, or a letter $a \in \mathcal{A}$.
 \item If $x \in \mathcal{A} \cup \mathcal{B}$, then every nontrivial component of $W(u)\ssm\{o,x\}$ contains either $\bar{s}$, $t$, or a letter $a \in \mathcal{A}$.
\end{enumerate}

For the first point, note that  Whitehead automorphisms with an acting letter from $\mathcal{A}$ preserve $A$, so cannot be reducing since $|u|$ is minimal. Hence no letter of $\mathcal{A}$ is a reducing letter of $u$. As $x_1$ is in a nontrivial component $C$ of $W(u)$ and is not reducing, $\bar{x}_1$ is also in $C$. Furthermore, as $x \in \mathcal{A}$ is also not reducing, $x$ is not a cut vertex of $C$. Hence there is a path from $x_1$ to $\bar{x}_1$ in $C\ssm x$. This is also a path in $W(u)\ssm\{o,x\}$, unless it crosses $o$, in which case it will either pass through $\bar{s}$ or $t$ first.

For the second point, we work towards a contradiction by supposing that $C$ is a nontrivial component of $W(u)$ that does not contain $o$ or an element of $\mathcal{A}$. Pick $x \in C$ such that $\bar{x} \not \in C$ (such a letter exists by Lemma~\ref{l:basic}). As $o \not \in C$ and no letters from $\mathcal{A}$ appear in $C$, the reducing automorphism $\phi(C,x)$ fixes $A$. This contradicts the minimality of $|u|$ with respect to such automorphisms.

The third point uses the same idea. If $x \in \mathcal{A}$ then $x$ is not reducing, so if $x$ is contained in a nontrivial component $C$ of $W(u)$, it is not a cut vertex of this component ,and $\bar{x} \in C$. Thus the result follows from Part 2. If $x \in \mathcal{B}$ and $C$ is a component of $W(u) \ssm x$ that doesn't contain $o$ or an element of $\mathcal{A}$, then $C$ is not a component of $W(u)$ by Part 2. Hence  there are edges between $x$ and $C$, and $\phi(C \cup x, x)$ reduces the length of $u$. However, as $o$ and $\mathcal{A}$ are contained in the complement of $C \cup x$, this automorphism fixes $A$, which is a contradiction.

The fourth point follows quickly from the third: the component $C$ of $W(u)\ssm x$ containing $o$ also contains $\bar{s}$ and $t$. Either $o$ is not a cut vertex of $C$ or $o$ separates $c$ into two pieces, one containing $\bar{s}$ and the other containing $t$.
\end{proof}

\section{Extending automorphisms of $\mathcal{F}_{(2)}$ to $\mathcal{AF}_N$} \label{s:extending}

The purpose of this section is to prove Theorem~\ref{t:f2_is_rigid}. We fix  $N \geq 3$, let $\mathcal{AF}_N$ denote the free factor graph of $F_N$, and let $\mathcal{F}_{(k)}$ denote the subgraph of $\mathcal{AF}_N$ spanned by factors of rank at most $k$. 
In the light of \cite{BB}, Theorem~\ref{t:f2_is_rigid} is a consequence of the following proposition.

\begin{proposition}\label{p:going_up_1} Let $N \geq 3$ and $k \geq 2$. Every automorphism of $\mathcal{F}_{(k)}$ extends to a unique automorphism of $\mathcal{AF}_N$.
\end{proposition}

Rather than proving Proposition~\ref{p:going_up_1} directly, we instead proceed one layer at a time. 
\begin{proposition}\label{p:going_up_2}
Let $N \geq 3$ and $k \geq 3$. Every automorphism of $\mathcal{F}_{(k-1)}$ extends to a unique automorphism of $\mathcal{F}_{(k)}$.
\end{proposition}

Note that Proposition~\ref{p:going_up_1} follows from Proposition~\ref{p:going_up_2} because $\F_{(N-1)}=\FA$. The rough idea of the proof of Proposition~\ref{p:going_up_2} is quite simple but the details require work. For this reason, we will give an outline of the proof now (referencing yet-to-be-qualified statements), and in the rest of the section we will provide the missing details. 

\begin{remark}\label{r:flag} We will sometimes  blur the distinction between the \emph{free factor complex} and its 1-skeleton, the \emph{free factor graph} (and likewise for their subgraphs/subcomplexes). The former is the  simplicial flag complex determined by the latter, so passing between them does not alter the group of automorphisms. Thinking about the higher-dimensional simplices will be beneficial in this section, whereas for the rigidity results elsewhere in the paper, it is easier to think of graphs (the 1-skeleta of the complexes). In keeping with this point of view, we adopt the convention that the link $\lk(v)$ of a vertex $v$ in a graph is the full subgraph spanned by vertices adjacent to $v$ (not just
the set of vertices).
\end{remark}

\begin{proof}[Proof of Proposition~\ref{p:going_up_2}.]
 We say an automorphism $\phi$ of $\mathcal{F}_{(k)}$ is \emph{rank-preserving} if $\phi(A)$ is the same rank as $A$ for all free factors $A$. The proof boils down to the following points:
 
 \begin{itemize}
 \item Every automorphism of $\mathcal{F}_{(k)}$ is rank-preserving.
 \item If $A$ and $A'$ are distinct rank $k$ factors, their links in $\mathcal{F}_{(k)}$ are distinct subgraphs of $\F_{(k-1)}$.
 \item The action of $\aut(\F_{(k-1)})$ permutes the set of such links.
 \end{itemize}
 
 With these three facts in hand, we see that every automorphism of $\mathcal{F}_{(k-1)}$ can be extended to
 an automorphism of $\mathcal{F}_{(k)}$ by recording the action on the set of links of rank $k$ factors: if $\phi(\lk(A))=\lk(A')$, extend $\phi$ by defining $\phi(A):=A'$. 
 
The first bullet point is covered by Proposition~\ref{p:rank}, whose proof is given below. The second bullet point is obvious, because each rank $k$ factor is generated (hence uniquely determined) by the factors that neighbour it in $\F_{(k-1)}$ (its lower link). The remaining work is concentrated on the third bullet point. The key idea is that if $A$ is a rank $k$ factor then its link in $\mathcal{F}_{(k-1)}$ is a union of \emph{standard $k$-apartments} associated to bases of $A$ (see Definition~\ref{d:standard_apartment} and Figure~\ref{f:apartment}). Standard $k$-apartments $\Delta$ are isomorphic to the barycentric subdivision of the boundary of a $k-1$-simplex, and are distinguished from other subcomplexes 
of this type by the fact that opposite vertices in $\Delta$ are antipodal (Proposition~\ref{p:standard_characterisation}). The Graded Antipode Lemma implies that this characterising property is preserved under (rank-preserving) automorphisms of $\mathcal{F}_{(k-1)}$, so the set of standard apartments is preserved under  automorphisms of $\mathcal{F}_{(k-1)}$ (Corollary~\ref{c:standard_preservation}). Finally, we argue that links themselves are preserved by using what one might call a \emph{crawling argument} (Proposition~\ref{p:crawling}):
 \begin{itemize}
 \item If two standard $k$-apartments share an antipodal pair of vertices generating a rank $k$ factor $A$ then  both
 apartments  lie in the link of $A$.
 \item Any two standard $k$-apartments in the link of a given rank $k$ factor $A$ can be connected by a chain of standard apartments such that each successive pair along the chain shares an antipodal pair of vertices generating $A$.
 \end{itemize}
 As antipodal pairs and standard apartments are preserved by automorphisms of $\F_{(k-1)}$, these facts tell us
 that the action of $\aut(\F_{(k-1)})$ permutes the set of rank-$k$ links. This completes the proof of the proposition. 
 \end{proof}

\begin{figure}[h]
\centering
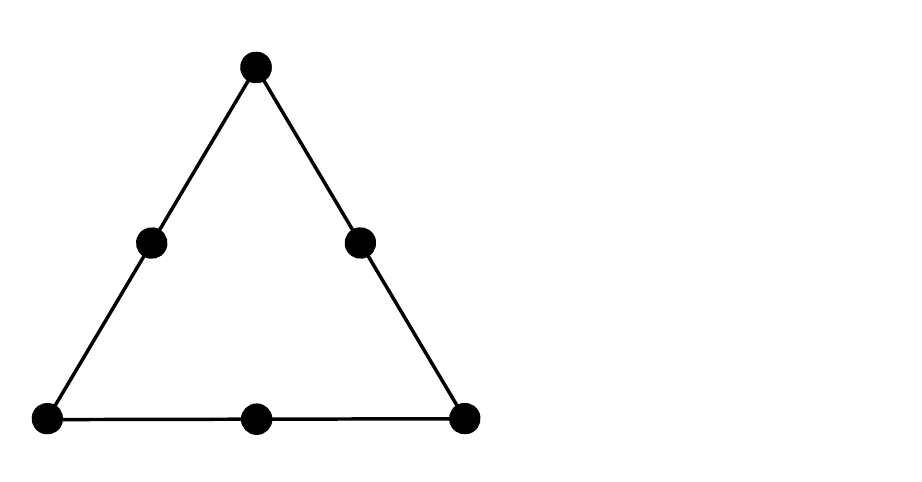
\caption{The standard 3-apartment given by a basis $\{x_1, x_2, x_3\}$ of a rank 3 factor.}\label{f:apartment} 
\end{figure} 
 
\begin{proposition}\label{p:rank}
For $2\le k\le N-1$, every automorphism of the graph $\F_{(k)}\subseteq \FA$ preserves the rank of vertices.
\end{proposition}

\begin{proof} The proof follows the proof of the case $k=N-1$ in \cite{BB}. If $A$ is a vertex of rank $i$ with $1<i<k$, then $\Lk(A)\subset \F_{(k)}$ is a join $\Lk_+(A)\ast \Lk_-(A)$, where $\Lk_+(A)\subset\Lk(A)$ is the subgraph spanned by neighbouring vertices of rank greater than $i$ and
$Lk_-(A)$ is the subgraph spanned by neighbouring vertices of rank less than $i$. Thus $\Lk(A)$
has (combinatorial) diameter $2$ unless $A$ has rank $1$ or $k$. In these cases,
$\Lk(A)$ has diameter greater than $2$; see  \cite[Lemma 3.2]{BB}.  

An additional argument is needed to distinguish rank one vertices from rank $k$ vertices. The key point to
observe here is that if $A$ is a rank $k$ vertex and $C$ is a rank one vertex, and $C$ is antipodal to
$A$, that is $\<A,C\> = A\ast C$ is a free factor of rank $k+1$, then $d(B,C)\leq 2$ for all vertices $B\in \Lk(A)$,
because $B\ast C<F_N$ is a free factor of rank at most $k$.  In contrast, given any rank $k$ vertex $A$ 
and a rank one vertex $C$ with $d(A,C)=3$, one can always find a rank two vertex $B$
adjacent to $C$ such that $A\cap B$ is trivial, meaning that $d(A,B)>2$; see \cite[Prop. 3.4]{BB}.

It remains to prove that an arbitrary automorphism $\phi$ of $\F_{(k)}$ preserves the rank of vertices 
of rank $1<r<k$. Let $B$ be such a vertex and let $A$ be a rank $k$ vertex containing $B$. As
$\autN$ acts transitively on the set of vertices of any given rank, we may compose $\phi$ with an 
element of $\autN$ to suppose that $\phi$ fixes $A$, in which case $\phi(B)$ is the image of $B$ under an
automorphism of $\Lk(A)$. There is an obvious rank-preserving
automorphism from this link to $\mathcal{AF}_k$, and any automorphism of the latter preserves rank \cite{BB}.
\end{proof} 
 
\subsection{Standard apartments of free factors}

We now get to work on justifying the third bullet point in the proof of Proposition~\ref{p:going_up_2}. We start with standard apartments and their invariance under $\aut(\mathcal{F}_{(k-1)})$. The arguments in this section follow the template
of the corresponding results in \cite{BB}.

\begin{definition}[Standard apartments and rank-isomorphisms] \label{d:standard_apartment} For $k\le N$,
a {\em standard $k$-apartment} $\Delta$  in $\mathcal{AF}_N$ is the full subcomplex whose vertices are the
free factors spanned by the non-empty proper subsets of  a basis 
for a  free factor $A<F_N$ of rank $k$ (see Figure~\ref{f:apartment}).
We say that a subcomplex $\Lambda \subset \mathcal{AF}_N$ is a {\em putative $k$-apartment} 
if there is a rank-preserving isomorphism from $\Lambda$ to a standard $k$-apartment.
\end{definition}

The following proposition tells us how to determine if a putative apartment is actually a standard apartment.

\begin{proposition}[Characterising standard apartments by their antipodes] \label{p:standard_characterisation}
Let  $\Lambda \subset \mathcal{F}_{(k-1)}$ be a putative $l$-apartment with $k\ge l \geq 3$. Suppose that for each rank 1 vertex $\langle u \rangle \in \Lambda$ and each vertex $A\in \Lambda$ of rank at least $2$, either $A$ and $u$ are adjacent (so $u \in A$) or $\langle A , u \rangle$ generates a free factor of rank ${\rm{rank}}(A) + 1$. Then $\Lambda$ is a standard $l$-apartment.
\end{proposition} 

\begin{proof}
We proceed by induction on $l$,
 starting with the case where $l =3$. In this case $\Lambda$ is a hexagon with alternating rank 1 and rank 2 vertices. Let $\langle x_1 \rangle, \langle x_2 \rangle, \langle x_3 \rangle$ be the rank 1 vertices and let $X_1, X_2, X_3$ be their opposite vertices. By hypothesis, $A= \langle x_1 \rangle \ast X_1$ is a rank 3 free factor, and $x_2$ is primitive in $X_1$, so $\langle x_1, x_2 \rangle$ is a rank 2 free factor. As $\langle x_1, x_2 \rangle < X_3$ and both are rank 2 free factors, $X_3 = \langle x_1, x_2 \rangle$. Similarly, $X_1 = \langle x_2, x_3 \rangle$ and $X_2 = \langle x_1, x_3 \rangle$. Thus $A=\langle x_1, x_2, x_3 \rangle$ and $\Lambda$ is the standard $3$-apartment for this basis of $A$.

We assume now that $\Lambda$ is a putative $l$-apartment with $l > 3$ and that the result holds for the boundary of every  facet of $\Lambda$, which means that each $(l-1)$-element subset of $l-1$ rank 1 vertices in $\Lambda$ generates a free factor of rank $l-1$, which much be equal to the free factor $A \in \Lambda$ adjacent to all of these vertices. If $\langle x_l \rangle$ is the remaining rank 1 vertex then by hypothesis $A \ast \langle x_l \rangle$ is a free factor of rank $l$.
The vertices of $\Lambda$ form a basis for this free factor. Thus
$\Lambda$ is standard. This completes the induction.
\end{proof}

\begin{corollary}[Standard apartments are invariant under $\aut(\mathcal{F}_{(k-1)})$] \label{c:standard_preservation}
For $3\le k \le N$, every automorphism of $\mathcal{F}_{(k-1)}$ sends standard $k$-apartments to standard $k$-apartments.
\end{corollary}

\begin{proof}
By the Graded Antipode Lemma, if $A$ is a vertex of $\mathcal{F}_{(k-1)}$ with rank at least 2 and $\langle u \rangle$ is a vertex of rank 1, then $A \ast \langle u \rangle$ is a free factor of $F_N$ if and only if $\langle u \rangle$ is not adjacent to $A$ and for every rank 1 vertex $\langle p \rangle$ adjacent to $A$ there exists a rank 2 vertex $B$ containing both $\langle p \rangle$  and $\langle u \rangle$. These conditions can be phrased purely in terms of distance  in $\mathcal{F}_{(k-1)}$
and rank, which Proposition \ref{p:rank} assures us is preserved by  $\aut(\mathcal{F}_{(k-1)})$. 
Proposition  \ref{p:standard_characterisation} completes the proof. \end{proof}

\subsection{Crawling through links}

We have shown that the set of standard $k$-apartments is preserved by automorphisms of $\mathcal{F}_{(k-1)}$. We now want to extend this argument to the set of links of rank $k$ factors.

\begin{proposition} \label{p:crawling}
Let $\Delta$ and $\Lambda$ be standard $k$-apartments  in $\mathcal{F}_{(k-1)}\subset \mathcal{F}_{(k)}$, with $k \geq 3$. Then, $\Delta$ and $\Lambda$ are in the link of the same rank $k$ factor in $\mathcal{F}_{(k)}$ if and only if there exists a sequence \[ \Delta=\Delta_0, \Delta_1, \ldots ,\Delta_n = \Lambda \] of standard $k$-apartments such that for all $i$ the intersection $\Delta_i \cap \Delta_{i+1}$ contains an antipodal pair of vertices $B_i$ and $\langle u_i \rangle$ of ranks $k-1$ and 1 respectively. 
\end{proposition}

\begin{proof}
If the stated condition holds then each apartment in the sequence is in the link of the same rank $k$ factor $A=B_i \ast \langle u_i \rangle$. We therefore focus on the converse assertion: we assume that $\Delta$ and $\Lambda$ are in the link of the same free factor $A$ and build a chain between them. The key fact to use here is that the subgroup 
$\aut(F_N, A)<\aut(F_N)$ that preserves $A$ acts transitively on the set of standard $k$-apartments in $\lk(A)$, since it acts transitively on the bases of $A$. Furthermore, if we choose a basis $x_1, \ldots, x_N$ of $F_N$ such that $A$ is generated by the first $k$ elements, then  $\aut(F_N, A)$  is generated by the inversions $x_i\mapsto x_i^{-1}$ and the right Nielsen automorphisms $\rho_{ij}:x_i\mapsto x_ix_j$  that preserve $A$ (this is well-known, e.g. \cite[Theorem~4.2]{MR3179663}). By induction on the word length of automorphisms with respect to this basis,  
we can reduce to the case where $\Lambda=\phi(\Delta)$ for some Nielsen automorphism or inversion 
preserving $A$ with $\Delta$ the standard $k$-apartment of the basis $\{ x_1, \ldots, x_k \}$. If $\phi$ is an inversion, then $\phi$ fixes $\Delta$ and we are done. The other possibility is $\phi=\rho_{ij}$ where either $i > k$ or $i,j \leq k$. If $i >k$ then $\rho_{ij}$ fixes $\Delta$ and we are done. Otherwise, choose $l \leq k$ not equal to either $i$ or $j$. Let $B$ be the factor generated by the basis elements of $A$ other than $x_l$. Then both $B$ and $\langle x_l \rangle$ are preserved by $\rho_{ij}$, so these vertices both lie in $\Delta$ and in $\Lambda=\rho_{ij}(\Delta)$. 
\end{proof}

\noindent{\bf{End of the Proof of Theorem \ref{t:f2_is_rigid}.}}
At the beginning of this section we reduced Theorem \ref{t:f2_is_rigid} to the three bullet points in the
statement of Proposition \ref{p:going_up_2}, and following Proposition \ref{p:rank} it only remained to prove
that automorphisms of $\mathcal{F}_{(k-1)}\subset \mathcal{F}_{(k)}$ preserved the set of links of rank $k$
vertices. We have proved that standard $k$-apartments and pairs of antipodal vertices are preserved by
all  automorphisms of $\mathcal{F}_{(k-1)}$, so the set of chains satisfying the conditions
of Proposition \ref{p:crawling} is also preserved. Thus links are preserved, as required, and we are done. \qed

\section{Constructing an action of $\Comm(\G)$ on $\mathcal{F}_{(2)}$.} \label{s:fat_action}

With \cite{BB} and the results of the previous section in hand, we will be able to establish commensurator-rigidity for
$\autN$  if we can extend the action of $\autN$ on $\mathcal{F}_{(2)}$
to an action by $\Comm(\autN)$. More generally, to prove Theorem B, we must extend the action of $\G$
to an action by $\Comm(\G)$ for every ample subgroup $\G < \autN$. That is our goal in this section.

\subsection{Direct products in ample subgroups of $\aut(F_N)$}\label{s:direct_products}

In \cite{BW}, we characterized the subgroups of $\Aut(F_N)$ and $\Out(F_N)$ that are direct products of free groups with the maximum number of factors.  A key example is obtained by taking a basis $\mathcal{B}=\{a_1,a_2, x_1,\dots,x_{N-2}\}$ of $F_N$ and defining
\begin{equation}\label{eq:standard1} D=L_1\times R_1
\times\dots\times L_{N-2}\times R_{N-2}\times I \end{equation}
where $I<\inn (F_N)$ is the group of inner automorphisms $\ad_w$ with 
$w\in\<a_1,a_2\>$ and
$L_i$  (resp. $R_i$)
is the group of automorphisms, fixing $x_j$ for $j\neq i$,
that have the form $x_i\mapsto w x_i$ (resp. $x_i\mapsto x_i w$)
with $w\in \<a_1,a_2\>$.  

We proved in \cite{BW}  that every direct product of $2N-3$ free groups in $\aut(F_N)$ is a variation on this example.
More precisely, for every such subgroup $D$   there exists a basis such that the direct factors of $D$ are contained in the groups \[ \langle L_1, \tau \rangle,\ldots, \langle L_{N-2}, \tau \rangle, \ldots, \langle R_{1}, \tau \rangle, \langle R_{N-2}, \tau \rangle, \langle I , \tau \rangle, \] where $\tau$ is the Nielsen automorphism  $a_1 \mapsto a_1 a_2$.
Writing $L_i^\tau$, $R_i^\tau$ and $I^\tau$ to denote the subgroups of $L_i, R_i$ and $I$ that commute with $\tau$, we
classified all of the possibilities for $D$ as follows.

\begin{theorem}[\cite{BW}, Theorem~7.1 and Proposition~7.3] \label{t:direct-products-aut}
Let $N \geq 3$. If $D<\aut(F_N)$ is a direct product of $2N-3$ nonabelian free groups, then a conjugate of $D$ is contained in one of the following groups:
\begin{itemize}
\item $L_1 \times \cdots \times L_{N-2} \times R_1 \times \cdots \times R_{N-2} \times I$
\item $L_1^\tau \times \cdots \times L_{N-2}^\tau \times R_1^\tau \times \cdots \times R_{N-2}^\tau \times I^\tau \times \langle \tau \rangle $
\item $\langle \tau, L_1 \rangle \times L_2^\tau \times \cdots L_{N-2}^\tau \times R_1^\tau \times \cdots R_{N-2}^\tau \times I^\tau$
\item $L_1^\tau \times \cdots \times L_{N-2}^\tau \times R_1^\tau \times \cdots \times R_{N-2}^\tau \times \langle I, \tau \rangle $
\end{itemize}
The centralizer of $D$ in $\aut(F_N)$ is trivial unless $D$ falls under the second case, in which case its centralizer is $\langle \tau \rangle$.
\end{theorem}

We shall refer to the unique factor contained in $\langle I , \tau \rangle$ as the \emph{inner factor} of $D$ and denote this $I(D)$. We say that a factor contained in $\langle L_i, \tau \rangle$ is \emph{untwisted} if it is contained in $L_i$ and is \emph{twisted} otherwise, and we use the same language for factors contained in subgroups of the form $\langle R_i, \tau \rangle$ and $\langle I , \tau \rangle$. We call a direct product of $2N-3$ nonabelian free groups \emph{untwisted} if every factor is untwisted, in which case there is a basis of $F_N$ such that the direct product is a subgroup of \eqref{eq:standard1}. It follows from Theorem \ref{t:direct-products-aut} that in every ample subgroup $\G$ of $\aut(F_N)$, the maximal (with respect to containment) untwisted direct products of $2N-3$ free groups are of the form:
\begin{equation}\label{eq:standard2} D=(L_1\cap \G) \times (R_1 \cap \G)
\times\dots\times(L_{N-2}\cap \G ) \times (R_{N-2} \cap \G )\times (I \cap \G), \end{equation} 
with inner factor $I(D)=I \cap \G$.

\subsection{Direct products under commensurations of ample subgroups}

We write $\mathcal{D}(\G)$ to denote the set of direct products of $2N-3$ nonabelian free groups in  a subgroup $\G < \aut(F_N)$. We partially order $\mathcal{D}(\G)$ by inclusion. Although this is not strictly needed in the work that follows, Corollary~8.2 of \cite{BW} tells us that every element of $\mathcal{D}(\G)$ is contained in a maximal element of $\mathcal{D}(\G)$.

\begin{lemma}\label{l:cent_fact}
Let $N \geq 3$. Suppose that $\G$ is an ample subgroup of $\aut(F_N)$ and let $D\in \mathcal{D}(\G)$. If $D_i < D$ is a direct factor that is twisted then there exists $D' \in \mathcal{D}(\G)$ that has $2N-4$ direct factors in common with $D$ and is such that the centralizer $C_\G(D')$   is nontrivial.
\end{lemma}

\begin{proof}
If $D_i$ is twisted, then we are not in the first case of Theorem~\ref{t:direct-products-aut}. Thus we can conjugate to ensure that all but at most one direct factor of $D$ is contained in a subgroup of $\langle L_i^\tau, \tau \rangle$, $\langle R_i^\tau, \tau \rangle$ or $\langle I^\tau, \tau \rangle$. As $\G$ is ample, some power of $\tau$ is contained in $\G$, so each factor
other than the possible exception has a  nontrivial centralizer in $\G \cap \langle \tau \rangle$, and
we obtain $D'$ by swapping in an appropriate copy of $L_i^\tau$, $R_i^\tau$, or $I^\tau$.
\end{proof}

Given the classification of centralizers in Theorem~\ref{t:direct-products-aut}, the following lemma is what
one intuitively expects. However, it requires a surprising amount of care to prove.

\begin{lemma}\label{l:cent_fact_2}
Let $N \geq 3$. Let $\G$ be an ample subgroup of $\aut(F_N)$ and let $D$ be a maximal, untwisted element of $\mathcal{D}(\G)$. If $D' \in \mathcal{D}(\G)$ has $2N-4$ direct factors in common with $D$ then the centralizer $C_{\aut(F_N)}(D')$ of $D'$ in $\aut(F_N)$ is trivial. \end{lemma}

\begin{proof} 
We split the proof into two cases:

\textbf{Case 1: $D$ and $D'$ have the same inner factor}

In this case the rank two free factor $A$ supporting the inner automorphisms of $D$ and $D'$ is the same. As $\G$ is ample, for every primitive element $p \in A$, the group $I(D)=I(D')$ contains some power of the inner automorphism $\ad_p$. In particular, viewed as a subgroup of $A$, the group $I(D')$ is not contained in the fixed subgroup of any Nielsen automorphism defined on $A$. It follows from Theorem~\ref{t:direct-products-aut} that the centralizer of $D'$ in $\aut(F_N)$ is trivial.

\textbf{Case 2: $D$ and $D'$ have the same non-inner factors}
Let $H$ be the direct product of the $2N-4$ non-inner factors of $D$ and $D'$. Then $H$ injects into $\out(F_N)$ under the map $\pi: \aut(F_N) \to \out(F_N)$. By Theorem~A of \cite{BW}, the group $\pi(H)$ fixes a unique rose with $N-2$ petals in the boundary of Outer space, so that $D$ and $D'$ both act on the Bass--Serre tree $T$ of this rose with unique global fixed points $v$ and $v'$ respectively (\cite[Theorem~B]{BW}). By Lemma~4.2 of \cite{BW}, the stabilizer of any arc in $T$ of length greater than or equal to two has product rank at most $2N-5$. As $H < D \cap D'$ has product rank $2N-4$, it follows that $v$ and $v'$ are either equal or separated by an edge $e$ in $T$. If $v=v'$, then the same basis of $F_N$ used to give a standard description of $D$ can be used for $D'$. As $\G$ is ample, the same argument as above shows that $H$ is not contained in a subgroup of the form $L_1^\tau \times \cdots \times L_{N-2}^\tau \times R_1^\tau \times \cdots \times R_{N-2}^\tau \times I^\tau \times \langle \tau \rangle $ for any Nielsen automorphism $\tau$, so the centralizer of $D'$ in $\aut(F_N)$ is trivial. If $v'$ is adjacent to $v$ then without loss of generality we can assume that $v'=x_1v$ and the basis used to represent $D'$ in its standard form is  $\{x_1a_1x_1^{-1},x_1a_2x_1^{-1}, x_1,x_1x_2x_1^{-1},\ldots,x_1x_{N-2}x_1^{-1}\}$. Elements of $R_2 \cap \Gamma$ in $D$ send $x_2 \mapsto x_2w$, with $w \in \langle a_1, a_2 \rangle$, fixing the other elements of $\{ a_1, a_2, x_1, \ldots, x_{N-2}\}$. With respect to the basis for $D'$, such an element of $R_2 \cap \Gamma$ sends $x_1x_2x_1^{-1} \mapsto x_1x_2x_1^{-1}\cdot x_1wx_1^{-1}$, fixing the other elements of the  basis. As $\G$ is ample, $R_2 \cap \Gamma$ cannot commute with any Nielsen automorphism defined on $x_1 A x_1^{-1}$, so again the centralizer of $D'$ in $\aut(F_N)$ is trivial.
\end{proof}

\begin{proposition}\label{p:transfer}
Let $N \geq 3$. If $f:\G_1 \to \G_2$ is an isomorphism between ample subgroups of $\aut(F_N)$ and $D \in \mathcal{D}(\G_1)$ is untwisted and maximal in $\mathcal{D}(\G_1)$, then $f(D)$ is untwisted and maximal in $\mathcal{D}(\G_2)$.
\end{proposition}

\begin{proof}
As $f$ is an isomorphism, it induces a containment-preserving map from $\mathcal{D}(\G_1)$ to $\mathcal{D}(\G_2)$, so if $D$ is maximal in $\mathcal{D}(\G_1)$ then $f(D)$ is maximal in $\mathcal{D}(\G_2)$. Furthermore, as $D$ is maximal and untwisted, if $D' \in \mathcal{D}(\G_1)$ is obtained from $D$ by swapping out one factor then $C_{\G_1}(D')$ is trivial (Lemma~\ref{l:cent_fact_2}). The same centraliser behaviour is true of $f(D)$,
so $f(D)$ must untwisted, otherwise we would obtain a contradiction to Lemma~\ref{l:cent_fact}.
\end{proof}

\subsection{Powers of Nielsen automorphisms in ample subgroups}

Powers of Nielsen automorphisms in ample subgroups of $\aut(F_N)$ can be characterized as follows.

\begin{proposition}
An automorphism $\phi$ in an ample subgroup $\G < \aut(F_N)$ is a power of a Nielsen automorphism if and only if $\phi$ centralizes a direct product of $2N-3$ nonabelian free groups.
\end{proposition}

\begin{proof}
By Theorem~\ref{t:direct-products-aut}, the only elements of $\aut(F_N)$ that centralize a direct product of $2N-3$ nonabelian free groups are powers of Nielsen automorphisms. Conversely, using the notation of the previous section, if  $\tau$ is the Nielsen automorphism  $a_1 \mapsto a_1a_2$ fixing $a_2, x_1, \ldots, x_{N-2}$, then $\tau$ centralizes  \[ (L_1^\tau \cap \G) \times \cdots \times (L_{N-2}^\tau \cap \G) \times (R_1^\tau \cap \G) \times \cdots \times (R_{N-2}^\tau \cap \G) \times (I^\tau \cap \G).  \] If $\G$ is ample then all of the above factors are nonabelian free groups.
\end{proof}

\begin{remark}\label{c:durr!} We began this section by describing the prototypical 
untwisted direct product $D<\autN$; it was displayed as \eqref{eq:standard1}.
The inner factor $I(D)$ is the unique direct factor of $D$ that does not contain a power of any Nielsen transformation. As an ample subgroup $\G$ contains a power of every Nielsen transformation, this property characterises
the maximal untwisted elements of $\mathcal{D}(\G)$.
\end{remark}

\begin{proposition}[Constructing an action on $\mathcal{F}_{(2)}$.] \label{p:factor_map}
Let $N \geq 3$. We identify $F_N$ with $\innn$. If $A < \innn$ is  a free factor of rank 1 or 2 and $f:\G_1 \to \G_2$ is an isomorphism between ample subgroups of $\autn$, then $f(A \cap \G_1)$ is a subgroup of $\innn$ contained in a unique minimal free factor $f_*(A)$ with the same rank as $A$. Furthermore:
\begin{itemize}
 \item If $f': \G_1' \to \G_2'$ is an isomorphism that agrees with $f$ on a common finite-index subgroup of their domains, then $f_*(A)=f'_*(A)$.
 \item If $g:\G_2 \to \G_3$ is a second isomorphism between ample subgroups, then $(gf)_*(A)=g_*(f_*(A))$.
 \end{itemize}
 \end{proposition}
 
\begin{proof} We prove these statements for rank 2 factors. The result for rank 1 factors follows by taking intersections, along with the fact that each ample subgroup contains a power of every primitive inner automorphism. If a nonabelian subgroup of $F_N=\innn$ is contained in a rank 2 free factor, then this free factor is unique, as intersections of free factors are also free factors (and therefore cyclic for distinct rank 2 factors). 

Let $A$ be a rank 2 free factor and let $D \in \mathcal{D}(\G)$ be a maximal untwisted subgroup such that $I(D)=A\cap\G$ (as in Equation~\eqref{eq:standard2}). Then $f(D)$ is maximal and untwisted in $\G_2$ by Proposition~\ref{p:transfer}.  Furthermore, the inner factor $I(D)$ of $D$ is sent to the inner factor $I(f(D))$ of $f(D)$ under $f$, by Remark~\ref{c:durr!} (all other factors of $D$ and $f(D)$ contain powers of Nielsen transformations, whereas $I(D)$ and $I(f(D))$ do not). Hence $f(I(D))=f(A\cap \G_1)$ is a nonabelian subgroup of a rank 2 free factor , which we define to be $f_*(A)$.

If $f'$ satisfies the conditions of the first bullet point then $f'(\G_1'\cap A)$ and $f(\G_1 \cap A)$ are contained in rank 2 free factors of $\innn$, namely $f'_*(A)$ and $f_*(A)$, and the intersection of these factors has finite index in each. As the intersection of distinct rank 2 free factors in $F_N$ is  cyclic, we conclude that $f_*(A)=f'_*(A)$.

Suppose that $g:\G_2 \to \G_3$ is a second isomorphism between ample subgroups. Then $(gf)_*(A)$ is the unique rank two free factor containing \[ g(f_*(A) \cap \G_2)=gf(A \cap \Gamma_1) \] and therefore $(gf)_*(A)=g_*(f_*(A))$.
\end{proof}
 
We now have enough tools in hand to construct the action of $\Comm(\G)$ that we need for Theorem \ref{t:comm_f2}.

\begin{proposition} \label{p:action}
Let $\G$ be an ample subgroup of $\aut(F_N)$ and assume $N\ge 3$. Given an isomorphism $f:\G_1 \to \G_2$ between two finite index subgroups of $\G$, the mapping $A \mapsto f_*(A)$ induces an action $\Phi: \Comm(\G) \to \aut(\F_{(2)})$ such that the diagram
\[ \xymatrix{\Comm_{\aut(F_N)}(\G) \ar@/^2pc/[rr]^{\Psi_0} \ar[r]^-{\text{{\rm{ad}}}} & \Comm(\G) \ar[r]^\Phi & \aut(\F_{(2)}) } \]
commutes, where $\Psi_0$ is the restriction of the natural action of $\aut(F_N)$ on $\F_{(2)}$.
\end{proposition}

\begin{proof}
Given $[f] \in \Comm(\G)$, Proposition~\ref{p:factor_map} constructs a map $f_*$ on the set of rank 1 and 2 free factors that is independent of the chosen representative $f \in [f]$. By construction $f_*$ preserves containment so is a graph (poset) map. The second bullet point in Proposition \ref{p:factor_map} assures us that $f_*$ is a bijection with inverse $f^{-1}_*$,
and that $[f]\mapsto f_*$ is a homomorphism. Thus we have an action  $\Phi: \Comm(\G) \to \aut(\F_{(2)})$.

 Let $\phi \in \Comm_{\aut(F_N)}(\G)$. Note that if $g \in F_N$ then \[ \phi \cdot \ad_g \cdot \phi^{-1}= \ad_{\phi(g)} \] which implies that \[(\ad_\phi)_*(A)=\phi(A) \] for any subgroup $A < F_N=\innn$.
 This implies that the above diagram commutes, or equivalently that action of $\Comm_{\aut(F_N)}(\G)$ on $\F_{(2)}$ extends to a rank-preserving action of $\Comm(\G)$.
\end{proof}

\noindent{\bf{End of the Proof of Theorem \ref{t:n2}.}} In Section 3 we proved that automorphisms of
$\F_{(2)}$ extend uniquely to $\FA$, giving an isomorphism $\aut(\F_{(2)})\cong \aut(\FA)$.
From \cite{BB} we know that the natural action of $\aut(F_N)$ defines an isomorphism $\Psi: \aut(F_N)\to \aut(\FA)$,
and in Proposition \ref{p:action} we saw that the restriction of $\Psi$ to the relative commensurator of any ample
subgroup $\G$ can be extended to an action of $\Comm(\G)$ on $\F_{(2)}$. With these results in hand,
the following standard proposition completes the proof of Theorem~\ref{t:n2}, taking into account centralizers of (finite-index subgroups of) ample subgroups in $\aut(F_N)$ are trivial (Lemma~\ref{l:centralizers}) \qed


\begin{proposition}\label{l:Ivanov} Suppose that $\G$ acts on a graph $X$  faithfully. If the action of $\G$ extends to an action $\Phi \co \Comm(\G) \to \Aut(X)$ of $\Comm(\G)$ via   $\ad \co \G \to \Comm(\G)$, then \begin{itemize}
		\item The map $\ad \co \G \to \Comm(\G)$ is injective.
		\item $\Phi$ is injective and its image is contained in $\Comm_{\Aut(X)}(\G)$.
	\end{itemize}
Furthermore, if the centralizer $C_{\aut(X)}(\G_1)$ is trivial for every finite-index subgroup $\G_1 < \G$, then $\ad$ extends to an isomorphism $$\ad \co \Comm_{\aut(X)}(\G) \to \Comm(\G).$$
\end{proposition}

\begin{proof}
The injectivity of $\ad$ follows from the fact that the action of $\G$ on $X$ is faithful and $\Phi\circ \ad $. Suppose $f \co \G_1 \to \G_2$ is an isomorphism between two finite index subgroups of $\G$ representing $[f] \in \Comm(\G)$. If $g \in \G_1$ then one can verify that \[ [f]\cdot[\ad_g]\cdot[f^{-1}]=[\ad_{f(g)}] \] in $\Comm(\G)$, therefore
\begin{equation}\label{e:iva} \Phi([f]) \cdot \Phi([\ad_g]) \cdot \Phi([f])^{-1} =\Phi([\ad_{f(g)}]) \end{equation}  in $\aut(X)$.  Equation~\ref{e:iva} tells us that $\Phi([f])$ is contained in the relative commensurator $\Comm_{\Aut(X)}(\G)$. Furthermore, if $[f]$ is nontrivial we can pick $g$ such that $g \neq f(g)$, so that $\Phi([\ad_g]) \neq \Phi([\ad_{f(g)}])$. In this case Equation~\eqref{e:iva} tells us that $\Phi([f])$ is nontrivial if $[f]$ is nontrivial, so $\Phi$ is injective. The map $\ad$ always extends to a map $$\ad \co \Comm_{\aut(X)}(\G) \to \Comm(\G)$$  on the relative commensurator, and by \eqref{e:iva}, the composition $\ad \circ \Phi$ is the identity on $\Comm(\G)$. It follows that $\ad$ is an surjective, and this map is injective if and only if $C_{\aut(X)}(\G_1)=1$ for every finite index subgroup $\G_1 <\G$.
\end{proof}

\begin{lemma}\label{l:centralizers}
	If $\G$ contains a power of every inner automorphism by a primitive element, then its centralizer $C_{\aut(F_N)}(\G)$ is trivial. In particular, (finite-index subgroups of) ample subgroups and finite-index subgroups of $\ian$ have trivial centralizers in $\aut(F_N)$.
\end{lemma}

\begin{proof}
	Let $\{x_1, \ldots, x_N\}$ be a basis of $F_N$ and suppose there exist $k_i \in \mathbb{N}_{>0}$ such that $(\ad_{x_i})^{k_i}  \in \Gamma$ for all $i$. An automorphism $\phi$ commutes with $(\ad_{x_i})^{k_i}$ if and only if $\phi$ fixes $x_i^{k_i}$, which happens if and only if $\phi$ fixes $x_i$, since roots are unique in free groups. Hence the only automorphism centralizing $(\ad_{x_i})^{k_i}$ for all $i$ is the identity.
\end{proof}

\section{Commensurations of the Torelli subgroup} \label{s:torelli}

Recall that $\ian$ is the subgroup of $\aut(F_N)$ that acts trivially on the homology of $F_N$. Pointing to the  analogous subgroup of the mapping class group, $\ian$ is often called the \emph{Torelli subgroup} of $\aut(F_N)$. As Nielsen automorphisms act nontrivially on homology, the Torelli subgroup is not an ample subgroup of $\aut(F_N)$. However, we can exploit the fact that maximal (with respect to inclusion) direct products in $\ian$ have a very special form in order to induce an action of $\Comm(\ian)$ on $\mathcal{F}_{(2)}$.

\begin{proposition}\label{p:maximal_products_in_torelli}
Let $\G$ be a finite-index subgroup of $\ian$. If $D \in \mathcal{D}(\G)$ is maximal with respect to inclusion, then there exists a basis $\{a_1,a_2,x_1 \ldots, x_{N-2}\}$ of $F_N$ such that, in the notation of Section~\ref{s:direct_products},  
\[ D=(L_1\cap \G) \times (R_1 \cap \G)
\times\dots\times(L_{N-2}\cap \G) \times (R_{N-2} \cap \G )\times (I \cap \G). \]
The inner factor of $D$ is finitely generated, but the other direct factors of $D$ are not.
\end{proposition}

\begin{proof}
For the first part of the statement it is enough to show that no factor of $D$ is twisted, as then $D$ is contained in a maximal untwisted direct product of $2N-3$ nonabelian free groups in $\aut(F_N)$. If there were a twisted factor, it would
contain an element   that was a product of a transvection (or inner automorphism in the case of the inner factor) multiplied by a non-zero power $\tau^k$ of the Nielsen automorphism $\tau$.
But such an automorphism $\phi$ would act nontrivially on the homology of $F_N$,  as $\phi(a_1)=a_1a_2^k$. It follows that twisting does not occur for direct products in $\ian$. Thus $D$ has no twisted factors.

For the second part of the statement, recall that $L_i\cong F_2$ is the subgroup of $\aut(F_N)$ consisting of automorphisms that map $x_i \mapsto wx_i$ for some $w \in \langle a_1, a_2 \rangle$, fixing the other basis elements. Such an automorphism acts trivially on homology if and only if $w$ is trivial in the abelianization of $F_N$. Thus $L_i \cap \ian$ is the commutator subgroup $[L_i, L_i]$, which is an infinitely generated free group. The same argument applies to each factor other
than the inner factor.  All inner automorphisms belong to $\ian$, so  $I \cap \ian = I\cong F_N$. As 
$I\cap\G$ has finite index in $I$, it too is finitely generated. 
\end{proof}

\begin{proposition} \label{p:factor_map_torelli}
If $A < \innn$ is a free factor of rank 1 or 2 and $f:\G_1 \to \G_2$ is an isomorphism between finite-index subgroups of $\ian$ then $f(A \cap \G_1)$ is contained in a unique minimal free factor $f_*(A) < \innn$ with the same rank as $A$. Furthermore:
\begin{itemize}
 \item If $f': \G_1' \to \G_2'$ is an isomorphism that agrees with $f$ on a common finite-index subgroup of their domains, then   $f_*(A)=f'_*(A)$.
 \item If $g:\G_2 \to \G_3$ is a second isomorphism between finite-index subgroups of $\ian$, then $(gf)_*(A)=g_*(f_*(A))$.
 \end{itemize}
 \end{proposition}
 
\begin{proof}
The proof is essentially the same as the proof  of Proposition~\ref{p:factor_map}: we focus on the rank 2 case then,
noting that every inner automorphism belongs to $\ian$, look at the intersections of these rank 2 factors (in $\ian$ or a finite-index subgroup) to deduce the rank 1 case. The only place where the argument differs is the justification of the following two points:
\begin{itemize}
\item An untwisted element $D \in \mathcal{D}(\G_1)$ is sent to an untwisted element $f(D) \in \mathcal{D}(\G_2)$.
\item If  $D \in \mathcal{D}(\G_1)$ is maximal and untwisted then the inner factor of $D$ is sent to the inner factor of $f(D) \in \mathcal{D}(\G_2)$.
\end{itemize}
Fortunately, each of these facts is easier to prove in the case of $\ian$ --- the first is vacuous, since \emph{all} direct products in $\ian$ are untwisted, and the second follows immediately from the last sentence of Proposition~\ref{p:maximal_products_in_torelli}.
\end{proof}

With Proposition \ref{p:factor_map_torelli} in  place of Proposition~\ref{p:factor_map}, we
obtain an action of $\Comm(\ian)$ on $\mathcal{F}_{(2)}$ by arguing as in Proposition~\ref{p:action}.

\begin{proposition} \label{p:action_torelli}
Given an isomorphism $f:\G_1 \to \G_2$ between two finite index subgroups of $\ian$, the mapping $A \mapsto f_*(A)$ induces a homomorphism $\Phi: \Comm(\G) \to \aut(\F_{(2)})$ such that the diagram
\[ \xymatrix{\Comm_{\aut(F_N)}(\G) \ar@/^2pc/[rr]^{\Psi_0} \ar[r]^-{\text{{\rm{ad}}}} & \Comm(\G) \ar[r]^\Phi & \aut(\F_{(2)}) } \]
commutes, where $\Psi_0$ is the restriction of the natural action $\aut(F_N)\to  \aut(\F_{(2)})$.
\end{proposition}

\noindent{\bf{End of the Proof of Theorem \ref{t:torelli}.}} 
The commensurator rigidity of $\ian$  follows from Proposition~\ref{p:action_torelli}
and the rigidity of the action of $\aut(F_N)$ on $\mathcal{F}_{(2)}$, upon application of Proposition~\ref{l:Ivanov} and Lemma~\ref{l:centralizers}.
\qed

\section{$\aut(F_2)$ and the twice-punctured torus}

In the introduction, we gave a terse proof  that the abstract commensurator of $\aut(F_2)$ is the extended mapping class group of the five-punctured sphere. The purpose of this section is to give a geometric explanation of this result, which is based on a study of the twice-punctured torus and its hyperelliptic involution.  

\begin{figure}[h]
\centering
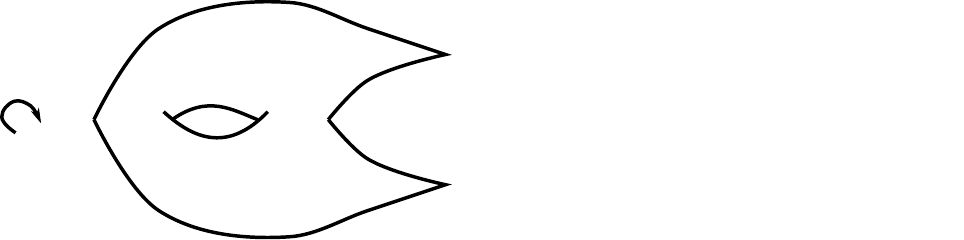
\caption{The twice-punctured torus $S_{1,2}$ has a hyperelliptic involution $\sigma$ that fixes four points and swaps the punctures. It induces an isomorphism from the pure mapping class group of $S_{1,2}$ to a subgroup of the mapping class group of the five-punctured sphere fixing one of the punctures.}\label{f:involution} 
\end{figure} 

Throughout,  $S_{g,p}$ will be a surface of genus $g$ with $p$ punctures (or marked points) and $\modd(S_{g,p})$ will denote its extended mapping class group: this is the group of isotopy classes of homeomorphisms $S_{g,p} \to S_{g,p}$. The (standard) mapping class group $\smod(S_{g,p})$ is the subgroup of index $2$ given by orientation-preserving homeomorphisms. The pure mapping class group  $\psmod(S_{g,p}) < \smod(S_{g,p})$ is the subgroup of orientation-preserving mapping classes that fix each puncture. This is an index two subgroup of the extended pure mapping class group $\pmodd(S_{g,p}) < \modd(S_{g,p})$, where the homeomorphisms must fix each puncture but are allowed to reverse orientation. When $ p \geq 1$ and $S_{g,p-1}$ is not a sphere (i.e. $p \geq 2$ or $g \geq 1$) then, by viewing one of the punctures as a basepoint, we obtain an injective  homomorphism \begin{equation*} \mu : \pmodd(S_{g,p}) \to \aut(\pi_1(S_{g,p-1})), \label{e:mu}\end{equation*} given by the action of the mapping class group on $\pi_1(S_{g,p-1})$. When $p \geq 2$, the fundamental group of the surface $S_{g,p-1}$ is free. In general, $\mu$ will not be surjective (see Remark~\ref{r:mu}), but it is an isomorphism in the case of a twice-punctured torus (where $\pi_1(S_{1,1}) \cong F_2$). Thus we obtain a geometric description of $\aut(F_2)$ as a mapping class group, which as far as we are aware, seems to have been overlooked in the literature.

\begin{proposition}[$\aut(F_2)$ is an extended pure mapping class group] \label{p:aut_tours}
Let $S=S_{1,2}$ be a twice-punctured torus. The map \[ \mu: \pmodd(S)\to \aut(F_2)\] is an isomorphism, which restricts to an isomorphism $\psmod(S)\to \Saut(F_2)$. Furthermore,
$$
\modd(S) \cong \aut(F_2) \times \Z/2\ \  \text{  and  }\ \ 
\smod(S) \cong \Saut(F_2) \times \Z/2,
$$
where the visible $\mathbb{Z}/2\mathbb{Z}$ is generated by a hyperelliptic involution $\sigma$.
\end{proposition}

\begin{proof}

It is convenient to specify a definite model for $S=S_{1,2}$. We view $S$ as the quotient of the tiling of the euclidean plane $\mathbb{E}^2$
by unit squares with integer coordinates,
with horizontal edges labelled $a$ and vertical edges labelled $b$. As the marked
points on $S$ we take the $\Z^2$-orbit $v$ of $(0,0)$ and the  $\Z^2$-orbit $c$  of $(1/2,1/2)$. In this model, the hyperelliptic involution $\sigma:S\to S$
is covered by the rotation through $\pi$ about the point  $(1/4,1/4)$. The four fixed points of $\sigma$ are represented by the
midpoints of the diagonals from the centre of any square to the corners -- i.e. the points with quarter-integer coordinates.

The extended pure subgroup  $\pmodd(S) < \modd(S)$ is the subgroup of index $2$ that fixes each of $v$ and $c$.
The central
cyclic subgroup of order $2$ generated by $\sigma$ is obviously a complement to $\pmodd(S)$ and hence 
$$\modd(S) = \pmodd(S) \times \<\sigma\>\cong \pmodd(S) \times \Z/2.$$
Similarly, $\smod(S)=\psmod(S)\times\<\sigma\>$.

Taking $v$ as a basepoint and puncturing $S$ at $c$,  we obtain an identification of $\pi_1(S\ssm\{c\},v)$ with the free group $F_2=F(a,b)$,
where the basis elements correspond to the labelling of the 1-cells in $S$. Then $S$ is a classifying space for $F_2$, and the natural homomorphism \[\mu:\pmodd(S)\to \aut(F_2)\] is injective \cite[Theorem~6.3]{MR214087}. We claim that it is also surjective.

Consider the orientation-preserving homeomorphism $\ell_1:S\to S$ covered by the  map $\E^2\to\E^2$ given on $[n,n+1]\times \R$
by $(n+t,y)\mapsto (n+t, y+n+2t)$ for $0\le t \le 1/2$ and $(n+t,y)\mapsto (n+t, y+n+1)$ for $1/2\le t\le 1$.
Recording the action of this map on the 1-cells of $S$, we see that the image of its mapping class under $\mu$
is  the Nielsen transformation $(\lambda_1: a\mapsto ba,\, b\mapsto b)$. Reversing the roles of $a$ and $b$,
we obtain a mapping class $[\ell_2]$ whose image under $\mu$ is
 $(\lambda_2: a\mapsto a,\, b\mapsto ab)$.

The homeomorphism $\overline{\iota}:S\to S$ covered by ${\rm{-I}}:\E^2\to\E^2$ is  orientation preserving and
$\mu$ sends its mapping class to $[a\mapsto a^{-1},\, b\mapsto b^{-1}]$. The image under $\mu$ of  
$\overline{\iota}\ell_1\overline{\iota}$ is the inverse of $(\rho_1: a\mapsto ab,\, b\mapsto b)$. Similarly, the image of
$\overline{\iota}\ell_2\overline{\iota}$ is the inverse of $(\rho_2: a\mapsto a,\, b\mapsto ba)$. 

As $\{\rho_1,\rho_2,\lambda_1,\lambda_2\}$ generate $\Saut(F_2)$, we have proved that $\mu(\psmod(S))$ contains $\Saut(F_2)$.
Conversely, the image of $\mu(\psmod(S))$
is contained in $\Saut(F_2)$ as elements of $\psmod(S)$ preserve orientation. To see that $\mu(\pmodd(S))$ is the whole of $\aut(F_2)$, note that the involution $r$ of $S$
covered by reflection in the line of slope $1$ through the origin in $\E^2$ acts on $\pi_1(S\ssm\{c\})$ by interchanging $a$ and $b$. 
\end{proof}

\begin{remark} \label{r:mu}
In the general case, the image of $\mu : \pmodd(S_{g,p}) \to \aut(\pi_1(S_{g,p-1}))$ is described in Theorem~5.7.1 of \cite{Z}. Taking account of orientation, there are $2(p-1)$ conjugacy classes in $\pi_1(S_{g,p-1})$ given by peripheral curves about the punctures (i.e. oriented curves bounding a punctured disk). The image of $\mu$ consists of the automorphisms that either fix each of these peripheral conjugacy classes, or else send each such conjugacy class to its inverse. In the case of $S_{1,1}$, the peripheral curves are conjugate to $[a,b]$ and its inverse, and every element of $\aut(F_2)$ sends $[a,b]$ to a conjugate of $[a,b]^{\pm1}$. \end{remark}

In the following lemma, which is well known to experts, the commutation is in the group of homeomorphisms of $S$ fixing $\{v,c\}$ -- in order words, the
maps are commuting exactly, not just up homotopy.

\begin{lemma}\label{l:symmetric}
Every element of $\modd(S)$ can be represented by a homeomorphism of $S\ssm\{v,c\}$ that commutes with the
hyperelliptic involution $\sigma$.
\end{lemma}

\begin{proof} In the course of the preceding proof, we proved that $\modd(S)$ is generated by the homotopy classes of the homeomorphisms
 $\{\sigma, \ell_1, \ell_2, \overline{\iota}, r\}$,
with a preferred lift of each to a map $f:\E^2\to \E^2$.
Each of the lifts that we described permutes the four classes of fixed points of the
hyperelliptic involution $\sigma$, i.e. the four classes of points with quarter-integer coordinates. With the exception of $r$ and $\sigma$
itself, we cannot choose a  specific lift of $\sigma$ that
 commutes with our preferred lifts $f:\E^2\to \E^2$, but for each we can choose a pair of lifts
$\tilde\sigma, \tilde\sigma'$ such that $ f \tilde\sigma = \tilde\sigma' f$; specifically, we can take $\tilde\sigma$ to be rotation about $(1/4, 1/4)$ and
$\tilde\sigma'$ to be  rotation about $f(1/4,1/4)$. Thus the induced homeomorphisms of $S$ commute in each case. Every homeomorphism
of $(S, \{v,c\})$ is homotopic ${\rm{rel} } \{v,c\}$ to a product of the generators 
 $\{\sigma, \ell_1, \ell_2, \overline{\iota}, r\}$ and their inverses, whence the desired representatives.
\end{proof}

Let $B_4$ denote the braid group on four strands, let $D_4$ be a four-punctured disc, 
and recall that $B_4 \cong \smod(D_4)$ (see \cite{Birman,FM}). We cap off the boundary of $D_4$ 
by attaching a punctured disc. This gives a left exact sequence:

\[ 1 \to \mathbb{Z} \to \smod(D_4) \to \smod(S_{0,5}).\]

The $\mathbb{Z}$ in this sequence is the Dehn twist around the boundary of the disc, 
which  is the centre $Z(B_4)$ of the braid group. The image of $B_4$ in $\smod(S_{0,5})$ is the subgroup 
consisting of mapping classes that fix a distinguished puncture (the `new' one that was added when capping off). 
Thus $B_4/Z(B_4)$ is isomorphic to this subgroup. It is also isomorphic to $\saut(F_2)$ according to the following proposition.

\begin{proposition}\label{p:last}
There is an injective homomorphism \[ f: \pmodd(S_{1,2}) \to \modd(S_{0,5}). \] The image of $f$ is the index-five subgroup of $\modd(S_{0,5})$ given by mapping classes fixing a distinguished puncture, namely the image of $\{v,c\}$ under the map $S_{1,2} \to S_{1,2}/\sigma \cong S_{0,5}$.
\end{proposition}

\begin{proof}
This follows from work of Birman--Hilden. In \cite{BH1,BH2}, they show that every \emph{symmetric homeomorphism} (i.e., one that commutes with $\sigma$ exactly) is isotopic to the identity only if there is an isotopy through symmetric homeomorphisms (see \cite{MW} for a modern account of such matters). It follows that two symmetric homeomorphisms of $S_{1,2}$ are isotopic if and only if the homeomorphisms of $S_{1,2}/\sigma \cong S_{0,5}$ that they induce via the quotient map are isotopic. Therefore the map $f$ is well-defined, as by Lemma~\ref{l:symmetric}, every mapping class in $\pmodd(S_{1,2})$ has a symmetric representative. To see that the image is as described in the proposition, first check that the reflection $r$ of $S_{1,2}$ described in Proposition~\ref{p:aut_tours} is still orientation-reversing after taking the quotient. This leaves us to understand  the image in $\smod(S_{0,5})$. The subgroup fixing the distinguished puncture is generated by half-twists about the remaining punctures (as one sees from the braid group picture) and one can check that these half-twists lift to Dehn twists on $S_{1,2}$ (see Figure~9.15 in \cite{FM}). 
\end{proof}

\begin{corollary}[Dyer--Formanek--Grossman, \cite{MR666911}]
The groups $\saut(F_2)$ and $B_4/Z(B_4)$ are isomorphic.
\end{corollary}

Combining Propositions \ref{p:aut_tours} and \ref{p:last}, we see that ${\rm{SAut}}(F_N)$ is isomorphic 
to the stabiliser in $\smod(S_{0,5})$ of a puncture, and the discussion before (\ref{p:last}) shows that $B_4/Z(B_4)$ is too.
Extending arguments of Ivanov, 
Korkmaz \cite[Theorem~3]{MR1696431} proved that $\Comm(\modd(S_{0,5})) \cong \modd(S_{0,5})$. 
This completes the promised geometric proof of Theorem \ref{t:comm_f2}, whose statement we recall for the
convenience of the reader.

\begin{theorem}
The abstract commensurator of $\aut(F_2)$ is the extended mapping class group of the five-punctured sphere, and $\aut(F_2)$ is an index five subgroup of its abstract commensurator.
\end{theorem}

\bibliographystyle{alpha}
\bibliography{commensurator-bib}

\begin{flushleft} 
Mathematical Institute\\
University of Oxford\\
Oxford OX2 6GG\\
\emph{e-mail: }\texttt{bridson@maths.ox.ac.uk, wade@maths.ox.ac.uk} 
\end{flushleft}

\end{document}

%% file: W1.pdf_tex
\begingroup%
  \makeatletter%
  \providecommand\color[2][]{%
    \errmessage{(Inkscape) Color is used for the text in Inkscape, but the package 'color.sty' is not loaded}%
    \renewcommand\color[2][]{}%
  }%
  \providecommand\transparent[1]{%
    \errmessage{(Inkscape) Transparency is used (non-zero) for the text in Inkscape, but the package 'transparent.sty' is not loaded}%
    \renewcommand\transparent[1]{}%
  }%
  \providecommand\rotatebox[2]{#2}%
  \newcommand*\fsize{\dimexpr\f@size pt\relax}%
  \newcommand*\lineheight[1]{\fontsize{\fsize}{#1\fsize}\selectfont}%
  \ifx\svgwidth\undefined%
    \setlength{\unitlength}{97.78051542bp}%
    \ifx\svgscale\undefined%
      \relax%
    \else%
      \setlength{\unitlength}{\unitlength * \real{\svgscale}}%
    \fi%
  \else%
    \setlength{\unitlength}{\svgwidth}%
  \fi%
  \global\let\svgwidth\undefined%
  \global\let\svgscale\undefined%
  \makeatother%
  \begin{picture}(1,0.7456781)%
    \lineheight{1}%
    \setlength\tabcolsep{0pt}%
    \put(0,0){\includegraphics[width=\unitlength,page=1]{W1.pdf}}%
    \put(-0.00318344,0.27750328){\color[rgb]{0,0,0}\makebox(0,0)[lt]{\lineheight{1.25}\smash{\begin{tabular}[t]{l}$o$\end{tabular}}}}%
    \put(0.30362612,0.00904474){\color[rgb]{0,0,0}\makebox(0,0)[lt]{\lineheight{1.25}\smash{\begin{tabular}[t]{l}$\bar{x}_1$\end{tabular}}}}%
    \put(0.53465738,0.01104704){\color[rgb]{0,0,0}\makebox(0,0)[lt]{\lineheight{1.25}\smash{\begin{tabular}[t]{l}$\bar{x}_2$\end{tabular}}}}%
    \put(0.30008366,0.65316016){\color[rgb]{0,0,0}\makebox(0,0)[lt]{\lineheight{1.25}\smash{\begin{tabular}[t]{l}$x_1$\end{tabular}}}}%
    \put(0.54035629,0.66378763){\color[rgb]{0,0,0}\makebox(0,0)[lt]{\lineheight{1.25}\smash{\begin{tabular}[t]{l}$x_2$\end{tabular}}}}%
    \put(0,0){\includegraphics[width=\unitlength,page=2]{W1.pdf}}%
    \put(0.7661509,0.66363325){\color[rgb]{0,0,0}\makebox(0,0)[lt]{\lineheight{1.25}\smash{\begin{tabular}[t]{l}$x_3$\end{tabular}}}}%
    \put(0.76260834,0.01104704){\color[rgb]{0,0,0}\makebox(0,0)[lt]{\lineheight{1.25}\smash{\begin{tabular}[t]{l}$\bar{x}_3$\end{tabular}}}}%
  \end{picture}%
\endgroup%

%% file: Stallings2.pdf_tex
\begingroup%
  \makeatletter%
  \providecommand\color[2][]{%
    \errmessage{(Inkscape) Color is used for the text in Inkscape, but the package 'color.sty' is not loaded}%
    \renewcommand\color[2][]{}%
  }%
  \providecommand\transparent[1]{%
    \errmessage{(Inkscape) Transparency is used (non-zero) for the text in Inkscape, but the package 'transparent.sty' is not loaded}%
    \renewcommand\transparent[1]{}%
  }%
  \providecommand\rotatebox[2]{#2}%
  \newcommand*\fsize{\dimexpr\f@size pt\relax}%
  \newcommand*\lineheight[1]{\fontsize{\fsize}{#1\fsize}\selectfont}%
  \ifx\svgwidth\undefined%
    \setlength{\unitlength}{255.66018172bp}%
    \ifx\svgscale\undefined%
      \relax%
    \else%
      \setlength{\unitlength}{\unitlength * \real{\svgscale}}%
    \fi%
  \else%
    \setlength{\unitlength}{\svgwidth}%
  \fi%
  \global\let\svgwidth\undefined%
  \global\let\svgscale\undefined%
  \makeatother%
  \begin{picture}(1,0.43729115)%
    \lineheight{1}%
    \setlength\tabcolsep{0pt}%
    \put(0,0){\includegraphics[width=\unitlength,page=1]{Stallings2.pdf}}%
    \put(-0.00121755,0.36861117){\color[rgb]{0,0,0}\makebox(0,0)[lt]{\lineheight{1.25}\smash{\begin{tabular}[t]{l}$x_1$\end{tabular}}}}%
    \put(-0.00121755,0.26593576){\color[rgb]{0,0,0}\makebox(0,0)[lt]{\lineheight{1.25}\smash{\begin{tabular}[t]{l}$x_2$\end{tabular}}}}%
    \put(-0.00121755,0.16326052){\color[rgb]{0,0,0}\makebox(0,0)[lt]{\lineheight{1.25}\smash{\begin{tabular}[t]{l}$x_3$\end{tabular}}}}%
    \put(-0.00121755,0.06058511){\color[rgb]{0,0,0}\makebox(0,0)[lt]{\lineheight{1.25}\smash{\begin{tabular}[t]{l}$x_4$\end{tabular}}}}%
    \put(0.41053155,0.36808738){\color[rgb]{0,0,0}\makebox(0,0)[lt]{\lineheight{1.25}\smash{\begin{tabular}[t]{l}$\bar{x}_1$\end{tabular}}}}%
    \put(0.41053155,0.26541197){\color[rgb]{0,0,0}\makebox(0,0)[lt]{\lineheight{1.25}\smash{\begin{tabular}[t]{l}$\bar{x}_2$\end{tabular}}}}%
    \put(0.41053155,0.16273673){\color[rgb]{0,0,0}\makebox(0,0)[lt]{\lineheight{1.25}\smash{\begin{tabular}[t]{l}$\bar{x}_3$\end{tabular}}}}%
    \put(0.41053155,0.06006132){\color[rgb]{0,0,0}\makebox(0,0)[lt]{\lineheight{1.25}\smash{\begin{tabular}[t]{l}$\bar{x}_4$\end{tabular}}}}%
    \put(0.39481596,0.41261489){\color[rgb]{0,0,0}\makebox(0,0)[lt]{\begin{minipage}{0.13201117\unitlength}\raggedright \end{minipage}}}%
    \put(0,0){\includegraphics[width=\unitlength,page=2]{Stallings2.pdf}}%
    \put(0.52787484,0.3680873){\color[rgb]{0,0,0}\makebox(0,0)[lt]{\lineheight{1.25}\smash{\begin{tabular}[t]{l}$x_1$\end{tabular}}}}%
    \put(0.52787484,0.2654118){\color[rgb]{0,0,0}\makebox(0,0)[lt]{\lineheight{1.25}\smash{\begin{tabular}[t]{l}$x_2$\end{tabular}}}}%
    \put(0.52787484,0.16273656){\color[rgb]{0,0,0}\makebox(0,0)[lt]{\lineheight{1.25}\smash{\begin{tabular}[t]{l}$x_3$\end{tabular}}}}%
    \put(0.52787484,0.06006115){\color[rgb]{0,0,0}\makebox(0,0)[lt]{\lineheight{1.25}\smash{\begin{tabular}[t]{l}$x_4$\end{tabular}}}}%
    \put(0.93962382,0.36756368){\color[rgb]{0,0,0}\makebox(0,0)[lt]{\lineheight{1.25}\smash{\begin{tabular}[t]{l}$\bar{x}_1$\end{tabular}}}}%
    \put(0.93962382,0.26488801){\color[rgb]{0,0,0}\makebox(0,0)[lt]{\lineheight{1.25}\smash{\begin{tabular}[t]{l}$\bar{x}_2$\end{tabular}}}}%
    \put(0.93962382,0.16221277){\color[rgb]{0,0,0}\makebox(0,0)[lt]{\lineheight{1.25}\smash{\begin{tabular}[t]{l}$\bar{x}_3$\end{tabular}}}}%
    \put(0.93962382,0.05953736){\color[rgb]{0,0,0}\makebox(0,0)[lt]{\lineheight{1.25}\smash{\begin{tabular}[t]{l}$\bar{x}_4$\end{tabular}}}}%
  \end{picture}%
\endgroup%

%% file: W2.pdf_tex
\begingroup%
  \makeatletter%
  \providecommand\color[2][]{%
    \errmessage{(Inkscape) Color is used for the text in Inkscape, but the package 'color.sty' is not loaded}%
    \renewcommand\color[2][]{}%
  }%
  \providecommand\transparent[1]{%
    \errmessage{(Inkscape) Transparency is used (non-zero) for the text in Inkscape, but the package 'transparent.sty' is not loaded}%
    \renewcommand\transparent[1]{}%
  }%
  \providecommand\rotatebox[2]{#2}%
  \newcommand*\fsize{\dimexpr\f@size pt\relax}%
  \newcommand*\lineheight[1]{\fontsize{\fsize}{#1\fsize}\selectfont}%
  \ifx\svgwidth\undefined%
    \setlength{\unitlength}{262.77895469bp}%
    \ifx\svgscale\undefined%
      \relax%
    \else%
      \setlength{\unitlength}{\unitlength * \real{\svgscale}}%
    \fi%
  \else%
    \setlength{\unitlength}{\svgwidth}%
  \fi%
  \global\let\svgwidth\undefined%
  \global\let\svgscale\undefined%
  \makeatother%
  \begin{picture}(1,0.35400506)%
    \lineheight{1}%
    \setlength\tabcolsep{0pt}%
    \put(0,0){\includegraphics[width=\unitlength,page=1]{W2.pdf}}%
    \put(-0.00118456,0.17979639){\color[rgb]{0,0,0}\makebox(0,0)[lt]{\lineheight{1.25}\smash{\begin{tabular}[t]{l}$o$\end{tabular}}}}%
    \put(0.11297982,0.07990249){\color[rgb]{0,0,0}\makebox(0,0)[lt]{\lineheight{1.25}\smash{\begin{tabular}[t]{l}$\bar{x}_1$\end{tabular}}}}%
    \put(0.19894696,0.08064755){\color[rgb]{0,0,0}\makebox(0,0)[lt]{\lineheight{1.25}\smash{\begin{tabular}[t]{l}$\bar{x}_2$\end{tabular}}}}%
    \put(0.11166166,0.31957898){\color[rgb]{0,0,0}\makebox(0,0)[lt]{\lineheight{1.25}\smash{\begin{tabular}[t]{l}$x_1$\end{tabular}}}}%
    \put(0.21367606,0.32009483){\color[rgb]{0,0,0}\makebox(0,0)[lt]{\lineheight{1.25}\smash{\begin{tabular}[t]{l}$x_2$\end{tabular}}}}%
    \put(0,0){\includegraphics[width=\unitlength,page=2]{W2.pdf}}%
    \put(0.39839079,0.17979639){\color[rgb]{0,0,0}\makebox(0,0)[lt]{\lineheight{1.25}\smash{\begin{tabular}[t]{l}$o$\end{tabular}}}}%
    \put(0.51255521,0.07990249){\color[rgb]{0,0,0}\makebox(0,0)[lt]{\lineheight{1.25}\smash{\begin{tabular}[t]{l}$\bar{x}_1$\end{tabular}}}}%
    \put(0.62030075,0.1322279){\color[rgb]{0,0,0}\makebox(0,0)[lt]{\lineheight{1.25}\smash{\begin{tabular}[t]{l}$\bar{x}_2$\end{tabular}}}}%
    \put(0.51123701,0.31957898){\color[rgb]{0,0,0}\makebox(0,0)[lt]{\lineheight{1.25}\smash{\begin{tabular}[t]{l}$x_1$\end{tabular}}}}%
    \put(0.61325145,0.32009483){\color[rgb]{0,0,0}\makebox(0,0)[lt]{\lineheight{1.25}\smash{\begin{tabular}[t]{l}$x_2$\end{tabular}}}}%
    \put(0,0){\includegraphics[width=\unitlength,page=3]{W2.pdf}}%
    \put(0.82650718,0.07990249){\color[rgb]{0,0,0}\makebox(0,0)[lt]{\lineheight{1.25}\smash{\begin{tabular}[t]{l}$\bar{x}_{k-1}$\end{tabular}}}}%
    \put(0.91247441,0.08064755){\color[rgb]{0,0,0}\makebox(0,0)[lt]{\lineheight{1.25}\smash{\begin{tabular}[t]{l}$\bar{x}_k$\end{tabular}}}}%
    \put(0.82518899,0.31957898){\color[rgb]{0,0,0}\makebox(0,0)[lt]{\lineheight{1.25}\smash{\begin{tabular}[t]{l}$x_{k-1}$\end{tabular}}}}%
    \put(0.92720342,0.32009483){\color[rgb]{0,0,0}\makebox(0,0)[lt]{\lineheight{1.25}\smash{\begin{tabular}[t]{l}$x_k$\end{tabular}}}}%
    \put(0,0){\includegraphics[width=\unitlength,page=4]{W2.pdf}}%
  \end{picture}%
\endgroup%

%% file: chamber2.pdf_tex
\begingroup%
  \makeatletter%
  \providecommand\color[2][]{%
    \errmessage{(Inkscape) Color is used for the text in Inkscape, but the package 'color.sty' is not loaded}%
    \renewcommand\color[2][]{}%
  }%
  \providecommand\transparent[1]{%
    \errmessage{(Inkscape) Transparency is used (non-zero) for the text in Inkscape, but the package 'transparent.sty' is not loaded}%
    \renewcommand\transparent[1]{}%
  }%
  \providecommand\rotatebox[2]{#2}%
  \newcommand*\fsize{\dimexpr\f@size pt\relax}%
  \newcommand*\lineheight[1]{\fontsize{\fsize}{#1\fsize}\selectfont}%
  \ifx\svgwidth\undefined%
    \setlength{\unitlength}{265.52127802bp}%
    \ifx\svgscale\undefined%
      \relax%
    \else%
      \setlength{\unitlength}{\unitlength * \real{\svgscale}}%
    \fi%
  \else%
    \setlength{\unitlength}{\svgwidth}%
  \fi%
  \global\let\svgwidth\undefined%
  \global\let\svgscale\undefined%
  \makeatother%
  \begin{picture}(1,0.52186118)%
    \lineheight{1}%
    \setlength\tabcolsep{0pt}%
    \put(0.42022261,0.27749261){\color[rgb]{0,0,0}\makebox(0,0)[lt]{\lineheight{1.25}\smash{\begin{tabular}[t]{l}$\langle x_1, x_2 \rangle$\end{tabular}}}}%
    \put(0.27774007,0.48752427){\color[rgb]{0,0,0}\makebox(0,0)[lt]{\lineheight{1.25}\smash{\begin{tabular}[t]{l}$\langle x_1 \rangle$\end{tabular}}}}%
    \put(0.50481505,0.01152588){\color[rgb]{0,0,0}\makebox(0,0)[lt]{\lineheight{1.25}\smash{\begin{tabular}[t]{l}$\langle x_2 \rangle$\end{tabular}}}}%
    \put(-0.00357614,0.01099357){\color[rgb]{0,0,0}\makebox(0,0)[lt]{\lineheight{1.25}\smash{\begin{tabular}[t]{l}$\langle x_3 \rangle$\end{tabular}}}}%
    \put(-0.00375146,0.27900627){\color[rgb]{0,0,0}\makebox(0,0)[lt]{\lineheight{1.25}\smash{\begin{tabular}[t]{l}$\langle x_1, x_3 \rangle$\end{tabular}}}}%
    \put(0.2365749,0.01065856){\color[rgb]{0,0,0}\makebox(0,0)[lt]{\lineheight{1.25}\smash{\begin{tabular}[t]{l}$\langle x_2, x_3 \rangle$\end{tabular}}}}%
    \put(0,0){\includegraphics[width=\unitlength,page=1]{chamber2.pdf}}%
  \end{picture}%
\endgroup%

%% file: involution.pdf_tex
\begingroup%
  \makeatletter%
  \providecommand\color[2][]{%
    \errmessage{(Inkscape) Color is used for the text in Inkscape, but the package 'color.sty' is not loaded}%
    \renewcommand\color[2][]{}%
  }%
  \providecommand\transparent[1]{%
    \errmessage{(Inkscape) Transparency is used (non-zero) for the text in Inkscape, but the package 'transparent.sty' is not loaded}%
    \renewcommand\transparent[1]{}%
  }%
  \providecommand\rotatebox[2]{#2}%
  \newcommand*\fsize{\dimexpr\f@size pt\relax}%
  \newcommand*\lineheight[1]{\fontsize{\fsize}{#1\fsize}\selectfont}%
  \ifx\svgwidth\undefined%
    \setlength{\unitlength}{277.53084678bp}%
    \ifx\svgscale\undefined%
      \relax%
    \else%
      \setlength{\unitlength}{\unitlength * \real{\svgscale}}%
    \fi%
  \else%
    \setlength{\unitlength}{\svgwidth}%
  \fi%
  \global\let\svgwidth\undefined%
  \global\let\svgscale\undefined%
  \makeatother%
  \begin{picture}(1,0.2482685)%
    \lineheight{1}%
    \setlength\tabcolsep{0pt}%
    \put(0,0){\includegraphics[width=\unitlength,page=1]{involution.pdf}}%
    \put(-0.01389044,0.15410795){\color[rgb]{0,0,0}\makebox(0,0)[lt]{\lineheight{1.25}\smash{\begin{tabular}[t]{l}$\sigma$\end{tabular}}}}%
    \put(0,0){\includegraphics[width=\unitlength,page=2]{involution.pdf}}%
  \end{picture}%
\endgroup%